\documentclass[12pt]{amsart}
\usepackage[normalem]{ulem}
\usepackage{amsmath}
\usepackage{amssymb}
\usepackage{mathrsfs}

\allowdisplaybreaks[2]

\usepackage[usenames, dvipsnames]{color}

\newcommand\set[1]{\{\,#1\,\}}
\newcommand\Iso{\operatorname{Iso}}

\newcommand\go{G^{(0)}}
\newcommand{\inv}{^{-1}}
\newcommand{\Z}{\mathbb{Z}}
\newcommand{\N}{\mathbb{N}}
\newcommand{\HoHo}{\pi}
\newcommand{\one}{\mathbf{1}}

\newcommand{\C}{\mathbb{C}}
\DeclareMathOperator{\supp}{supp}
\DeclareMathOperator{\lsp}{span}

\DeclareMathOperator{\dom}{Dom}
\DeclareMathOperator{\ran}{Ran}

\newtheorem{thm}{Theorem}[section]

\newtheorem{lemma}[thm]{Lemma}
\newtheorem{prop}[thm]{Proposition}
\newtheorem{cor}[thm]{Corollary}

\theoremstyle{definition}
\newtheorem{definition}[thm]{Definition}

\theoremstyle{remark}
\newtheorem{remark}[thm]{Remark}

\newtheorem{example}[thm]{Example}

\numberwithin{equation}{section}

\begin{document}

\title{Diagonal-preserving ring $*$-isomorphisms of Leavitt path algebras}

\author[J.H. Brown]{Jonathan H. Brown}
\address[J.H. Brown]{
Department of Mathematics\\
University of Dayton\\
300 College Park Dayton\\
OH 45469-2316 USA} \email{jonathan.henry.brown@gmail.com}

\author[L.O. Clark]{Lisa Orloff Clark}

\author[A. an Huef]{Astrid an Huef}

\address[L.O. Clark and A. an Huef]{Department of Mathematics and
Statistics, University of Otago, P.O. Box 56, Dunedin 9054, New Zealand}
\email{lclark@maths.otago.ac.nz}
\email{astrid@maths.otago.ac.nz }


\thanks{This research was supported by a University of Otago Research Grant and by a University of Dayton Research Institute SEED grant. We thank 
Aidan Sims for pointing us to Stone Duality, and for an observation   in \S\ref{section-Stone} which made our results in \S\ref{sec-main} much cleaner.}

\date{1 March  2016}

\begin{abstract} The graph groupoids of directed graphs $E$ and $F$ are topologically isomorphic if and only if 
there is a diagonal-preserving ring $*$-isomorphism between the Leavitt path algebras of  $E$ and $F$. 
\end{abstract}

\maketitle

\section{Introduction}

Brownlowe, Carlsen and Whittaker show that there is a diagonal-preserving isomorphism between two graph $C^*$-algebras  $C^*(E)$ and $C^*(F)$ if and only 
if the corresponding graph groupoids  are topologically isomorphic \cite[Theorem~5.1]{BCW}.  In this paper, we prove the analogous theorem for Leavitt path algebras over an integral domain $R$.
Diagonal-preserving isomorphisms of $C^*$-algebras of graphs and groupoids have recently been of interest in \cite{NR, BCW,  BNRSW}, and are starting to show up now in the theory of Leavitt 
path algebras \cite{JS, Toke,  ABHS, TRS}.  Indeed, we have  revised a previous version  \cite{BCaH-previous} of this paper to compare our results to  the ones just announced in \cite{Toke}.

It follows from  \cite[Theorem~4.2]{Ren}, applied to the graph groupoid, that the ``diagonal subalgebra'' $D_E$ of $C^*(E)$ is a maximal abelian subalgebra if and only if the graph groupoid is topologically principal  (if and only if  every cycle in $E$ has an entry). In that case, the Weyl groupoid, constructed by Renault in \cite{Ren} from the inclusion of $D_E$ in $C^*(G)$, is isomorphic to the graph groupoid.   In \cite{BCW}, Brownlowe, Carlsen and Whittaker extend  Renault's construction so that it applies to the inclusion of  $D_E$ in $C^*(E)$ when the graph may have cycles without entries, and  they call the resulting groupoid  the ``extended Weyl groupoid''. They prove that the extended Weyl groupoid is always topologically isomorphic to the graph  groupoid.  Then the strategy of the proof of the main theorem in \cite{BCW} is to show that if there is a diagonal-preserving isomorphism between the graph $C^*$-algebras, then the corresponding extended Weyl groupoids are isomorphic, and then so are the 
corresponding graph groupoids.

Let  $\HoHo: C^*(E)\to C^*(F)$ be a diagonal-preserving isomorphism of $C^*$-algebras.  The diagonal $D_E$ is isomorphic to the algebra $C_0(E^\infty)$  of continuous functions on the infinite-path space $E^\infty$ which vanish at infinity.  Since $\HoHo$ preserves the diagonal, spectral theory gives a homeomorphism $\kappa: E^\infty\to F^\infty$ between the infinite-path spaces such that 
\begin{equation}\label{eq:kappa}
\HoHo(d)\circ\kappa=d
\end{equation} 
for all $d\in C_0(E^\infty)$. 

Now let $\HoHo: L_R(E)\to L_R(F)$  be a ring $*$-isomorphism  between the Leavitt path algebras, and assume 
$\pi$ preserves the algebraic diagonal subalgebra. The idempotents in each diagonal form a Boolean algebra, and  $\pi$ restricts to give an isomorphism of these Boolean algebras.  We use Stone duality  to show that this gives a  concrete homeomorphism $\kappa$ between the infinite path spaces such that  \eqref{eq:kappa} holds 
for characteristic functions in the algebraic diagonal. We then use an algebraic version of the extended Weyl groupoid to show that the corresponding graph groupoids are isomorphic. 
It then follows that  for any commutative ring $S$ with identity, the Leavitt path algebras $L_S(E)$ and $L_S(F)$ are $*$-isomorphic as $S$-algebras. It also follows that the graph algebras $C^*(E)$ and $C^*(F)$ are isomorphic as $C^*$-algebras.

Taking $R=\C$  gives some evidence that the ``isomorphism conjecture for graph algebras'' from \cite[page~3752]{AT} may be true.  The conjecture says: if $L_\C(E)$ and  $L_\C(F)$ are isomorphic as rings, then $C^*(E)$ and $C^*(F)$ are isomorphic as $C^*$-algebras. 
Our Corollary~\ref{cor1} says that if there is a diagonal-preserving ring $*$-isomorphism of $L_\C(E)$ onto $L_\C(F)$, then 
the $C^*$-algebras are isomorphic.  
Further, in Corollary~\ref{cor2} we combine our results with \cite[Theorem~1]{Toke} to see that $L_\Z(E)$ and $L_\Z(F)$ are $*$-isomorphic as $\Z$-algebras if and only if there is 
a diagonal-preserving ring $*$-isomorphism of $L_R(E)$ onto $L_R(F)$.

We make two technical innovations.  An $n\in C^*(E)$ is a  normalizer of the diagonal if $n^*D_En\subset D_E$ and $nD_En^*\subset D_E$.  Both the Weyl groupoid and its extension are defined using a partial action $\alpha$ of the set of normalizers on the infinite path space $E^\infty$.  This $\alpha$ is constructed using spectral theory in \cite[\S1]{Kum} and in \cite[\S3]{Ren}.
Our first innovation is that we are able to define an analogous partial action, also called $\alpha$, of the normalizers  in $L_R(E)$ of the algebraic diagonal. The key result in the definition of our $\alpha$ is Proposition~\ref{prop:supp n}, and we show  $\alpha$ is a partial action in Proposition~\ref{prop:alg act new}. Interestingly,  we have been unable to extend Proposition~\ref{prop:supp n} from graph groupoids to more general groupoids, not even to those  of a higher-rank graph. 

The second  innovation is the construction of our algebraic version of the extended Weyl groupoid. In extending Renault's construction,  Brownlowe, Carlsen and Whittaker needed to deal with an isolated infinite path $x$, and they did so using the K-theory class of certain unitaries in a corner of $C^*(E)$ determined by $x$.   We noticed that we could replace the $K_1$ class of the unitary by its  degree as an element in the graded algebra $L_R(E)$.    (See Definition~\ref{def:eqrel} and Remark~\ref{rem:eqrel}).

Given an ample Hausdorff groupoid, 
Clark, Farthing, Sims and Tomforde  in \cite{CFST}, and independently Steinberg in \cite{Ste}, 
introduced an algebraic analogue of the groupoid $C^*$-algebra which is now called the Steinberg algebra.  
If the groupoid is the graph groupoid of a directed graph $E$, then the associated Steinberg algebra is isomorphic to $L_R(E)$. 
In our work it was helpful to use groupoid techniques and to view the elements of $L_R(E)$ as functions. For this reason, we view $L_R(E)$ as the 
Steinberg algebra of the graph groupoid throughout (see Theorem~\ref{thm-model}).

Related but disjoint to our work are the recent results of Ara, Bosa, Hazrat and Sims in \cite{ABHS}.  Given a graded ample Hausdorff groupoid with a topologically principal $0$-graded component, they show how to recover the groupoid from the graded-ring structure of the Steinberg algebra and a diagonal subalgebra.  For graphs in which every cycle has an entry, they 
deduce that a diagonal-preserving ring isomorphism of Leavitt path algebras implies that the underlying graph groupoids are isomorphic.

We have been careful to make  our paper accessible to those not so familiar with topological groupoids.
In particular, we have been careful to supply the technical details  when it comes to the construction of the algebraic version of the  Weyl groupoid and results involving it.

\section{Preliminaries}\label{sec-prelim}

\subsection{Directed graphs and Leavitt path algebras}

A directed graph $E=(E^0, E^1, r, s)$ consists of countable sets $E^0$ and $E^1$, and range and source maps $r,s: E^1\to E^0$.  We think of $E^0$ as the set of vertices and $E^1$ as the set of edges directed by $r$ and $s$.    A graph is {\em row-finite} if $|r\inv(v)|<\infty$ and   it {has no sources} if $|r\inv(v)|>0$. Throughout we assume that all graphs are row-finite and have no sources.

  We  use the convention that a path is  a sequence of edges $\mu=\mu_1\mu_2\cdots$ such that $s(\mu_i)=r(\mu_{i+1})$.  We  denote the $i$th edge in a path $\mu$ by $\mu_i$.  We say a path $\mu$ is finite if the sequence is finite and denote its length by $|\mu|$.  We denote the set of finite paths by $P(E)$ and 
the set of infinite paths by $E^\infty$. 
For $n\in \N$, we write $E^n$ for the set of paths in $P(E)$ of length $n$, and $vE^n$ to be the paths in $E^n$
with range $v$.  We consider the vertices $E^0$ to be paths of length $0$.
For $\mu\in P(E) \cup E^\infty$ and  $0<n < m< |\mu|$  we set
 \[
 \mu(n,m)=\mu_{n+1}\cdots \mu_m\quad\text{ and}\quad\sigma^m(\mu)=\begin{cases} \mu_{m+1}\cdots \mu_{|\mu|} & \text{if~}  \mu\in P(E)\\
  \mu_{m+1}\cdots &\text{if~}\mu\in E^\infty.
\end{cases}\]

We extend the map $r$ to $\mu\in P(E) \cup E^\infty$  by $r(\mu)=r(\mu_1)$;  for $\mu\in P(E)$ we also extend $s$ by $s(\mu)=s(\mu_{|\mu|})$. We say $\mu\in P(E)$ is a {\em cycle} if $r(\mu)=s(\mu)$ and 
$s(\mu_i)\neq r(\mu_j)$ for  $1<i,j<|\mu|$, $i\neq j$.  We say an edge $e$ is an {\em entrance} to a cycle  
$\mu$ if there exists $i$ such that $r(e)=r(\mu_i)$ and $e\neq \mu_i$: 
note that $e$ may equal $\mu_j$ for some $j\neq  i$.  
We say $\mu$ is a cycle without entrance if it has no entrance.

Let  $(E^1)^*:= \{e^*: e \in E^1\}$ be a set of formal symbols called {\em ghost edges}. If $\mu\in P(E)$, then   we write $\mu^*$ for $\mu_{|\mu|}^*\dots\mu_2^*\mu_1^*$ and call it a \emph{ghost path}.  We extend $s$ and $r$ to the ghost paths  by $r(\mu^*)=s(\mu)$ and $s(\mu^*)=r(\mu)$.

\begin{definition}
Let  $E $ be a row-finite  directed graph with no sources, let $R$ be a commutative ring with 1 and  let $A$ be an $R$-algebra.  A Leavitt $E$-family in $A$ is a set $\{P_v,S_e, S_{e^*}:  v\in E^0,e\in E^1\}\subset A$ 
where $\{P_v: v\in E^0\}$ is a set of mutually orthogonal idempotents, and
\begin{enumerate}
\item[(L1)]\label{L1} $P_{r(e)}S_e=S_e=S_eP_{s(e)}$;
\item[(L2)]\label{L2} $P_{s(e)}S_{e^*}=S_{e^*}=S_{e^*}P_{r(e)}$;
\item[(L3)]\label{L3} $S_{e^*}S_f=\delta_{e,f}  P_{s(e)}$; 
\item[(L4)]\label{L4} for all $v\in E^0$ we have $P_v=\sum_{e\in  vE^1} S_eS_{e^*}$.
\end{enumerate}
\end{definition}
For a path $\mu\in P(E)$ we set $S_\mu:=S_{\mu_1}\dots S_{\mu_{|\mu|}}$.
The \emph{Leavitt path algebra}  $L_R(E)$  is the universal $R$-algebra generated by a  
universal Leavitt $E$-family  $\{p_v,s_e, s_{e^*}\}$:  that is, if $A$ is an $R$-algebra and $\{P_v, S_e, S_{e^*}\}$ is a Leavitt 
$E$-family in $A$, then there exists a unique $R$-algebra homomorphism $\pi: L_R(E)\to A$ such that $\pi(p_v)=P_v$ and $\pi(s_e)=S_e$  \cite[\S2]{Tom}.  For the existence of such an $R$-algebra, see \cite[Theorem~3.4]{ACaHR}. 
It follows from (L3) that \[
L_R(E)=\lsp_R\{s_\mu s_\nu^*:\mu, \nu\in P(E)\}.
\]
Thus each element of $L_R(E)$ is of the form $\sum_{(\mu,\nu)\in F} r_{\mu,\nu} s_{\mu}s_{\nu^*}$ where $F$ is a finite subset of $\{(\mu,\nu)\in P(E)\times P(E): s(\mu)=s(\nu)\}$ and each $r_{\mu,\nu}\in R$.
We assume that  $R$ has an involution $r\mapsto \overline{r}$  which fixes 
both the multiplicative and additive identities. For subrings $R$ of $\C$ we 
use complex conjugation. There is then a natural involution on $L_R(E)$ given by
$\sum_{(\mu,\nu)\in F} r_{\mu,\nu} s_{\mu}s_{\nu^*}\mapsto \sum_{(\mu,\nu)\in F} \overline{r_{\mu,\nu}} s_{\nu}s_{\mu^*}
$. The \emph{diagonal}  of $L_R(E)$ is the commutative $*$-subalgebra \[
D(E):=\lsp_R\{s_\mu s_\mu^*:\mu\in P(E)\}.
\]

\subsection{Groupoids}

Informally, a groupoid is a generalization of a group where  multiplication is partially defined.  More formally, a \emph{groupoid} $G$ is a small category in which every morphism is invertible.  We identify the objects in $G$ with the identity morphisms, and call the set of objects the unit space $\go$ of $G$.    As $G$ is a category, each morphism $\gamma$ in $G$ has a range and a source, which we denote $r(\gamma)$ and $s(\gamma)$ respectively;  we consider $r$ and $s$ as maps $r,s: G\to \go$.\footnote{We are now using $r$ and $s$ to mean two different things:  the range and source maps 
in the graph $E$ and the range and source maps of the groupoid $G$.  
It should be clear from context which map is used.}   Notice that
$
r|_{\go}=\text{id}_{\go}=s|_{\go}
$. 
Morphisms $\gamma$ and $\eta$ in $G$ are composable if and only if $s(\gamma)=r(\eta)$.  
For subsets $B$ and $C$ of $G$ we set
\[
BC:=\{\gamma\eta: s(\gamma)=r(\eta), \gamma\in B, \eta\in C\}\quad\text{and}\quad B\inv:=\{\gamma\inv: \gamma\in B\}.
\]

Let $E$ be a row-finite directed graph with no sources. 
We now describe the graph groupoid $G_E$ of $E$ from \cite[Proposition~2.6]{KPRR97}.
Two infinite paths
$x_1x_2\dots$ and $y=y_1y_2\dots$ in $E$ are \emph{shift equivalent with lag $k$}, written $x\sim_k y$,  if there exists $N\in\N$ such that for $i\geq N$, $x_{i+k}=y_i$: that is the infinite paths eventually are the same modulo a shift of $k$ edges.  
This happens if and only if there exist $\mu,\nu\in P(E)$ and  $z\in E^\infty$ such that $|\mu|-|\nu|=k$, $x=\mu z$  and $y=\nu z$.    
A quick computation shows that if $x\sim_k y$ and $y\sim_l z$, then $x\sim_{k+l} z$.

The groupoid $G_E$  has objects $E^\infty$ and morphisms
\[
 G_E=\{ (x,k,y)\in E^\infty\times \Z\times E^\infty: x\sim_k y\}
\]  where $r(x,k,y)=x$ and  $s(x,k,y)=y$. Composition of morphisms is  given by
\[
(x,k,y)(y,l,z)=(x,k+l,z).
\]
The identity morphism at $x\in E^\infty$ is  $(x,0,x)$, and  we identify  $E^\infty$ with the identity morphisms via $x\mapsto (x,0,x)$.
If $(x,k,y)\in G_E$, then $(y,-k,x)\in G_E$, and $(x,k,y)(y,-k,x)=(x,0,x)$ implies that $(x,k,y)^{-1}=(y,-k,x)$.
 
Let $\mu,\nu\in P(E)$ with $s(\mu)=s(\nu)$. Set
\[
Z(\mu,\nu):=\{(\mu z, |\mu|-|\nu|, \nu z): z\in s(\mu)E^\infty\}.
\]
Then $Z(\mu,\nu)\subset G_E^{(0)}$ if and only if $\mu=\nu$; then $Z(\mu,\mu)=\{(\mu z, 0, \mu z): z\in s(\mu)E^\infty\}$, and we identify it with
 \[
 Z(\mu):=\{\mu z: z\in s(\mu)E^\infty\}.
 \]
The sets $Z(\mu,\nu)$ form a basis for a Hausdorff topology on $G_E$ such that multiplication and inversion
are continuous, and each $Z(\mu,\nu)$ is compact.  Further, 
$r(Z(\mu, \nu)) = Z(\mu)$ and $s(Z(\mu,\nu))= Z(\nu)$.   
(See \cite[Corollary~2.2]{KPRR97} for details.)

\subsection{Modelling $L_R(E)$ with a Steinberg algebra}

Let $E$ be a row-finite directed graph with no sources, $G_E$ the graph groupoid and $R$  a commutative ring with identity.
For $B \subset G_E$, let $\one_B: G_E\to R$ be the characteristic function of $B$. The \emph{Steinberg algebra} of $G_E$ is 
\[A_R(G_E) = \lsp\{\one_{Z(\mu,\nu)}:\mu, \nu \in P(E)\},\]
with pointwise addition and scalar multiplication, and 
convolution product and  involution given by
\[
f*g(\gamma)=\sum_{r(\eta)=r(\gamma)} f(\eta)g(\eta\inv\gamma)\text{\ and\ }
f^*(\gamma)=\overline{f(\gamma\inv)}
\]
for $f,g \in A_R(G)$.   The support $\supp(f)={\set{ \gamma:f(\gamma)\neq 0}}$ of  $f \in A_R(G_E)$ is compact open in $G_E$. We have 
$\supp(f*g) \subset \supp(f)\supp(g)$.  By   \cite[Lemma~4.2]{ACaHR} or by \cite[Lemma~3.6]{CFST} we can write each $f\in A_R(G)$ in  \emph{normal form}
\[f=\sum_{(\mu,\nu)\in F} r_{\mu,\nu}\one_{Z(\mu,\nu)}\] where $F$ is a finite subset of $\{(\mu,\nu)\in P(E)\times P(E): s(\mu)=s(\nu)\}$  and $(\mu,\nu), (\beta,\gamma)\in F$ implies $|\nu|=|\beta|$ and $0\neq r_{\mu,\nu}\in R$;  it follows that the $Z(\mu,\nu)$ in $F$ are mutually disjoint. 
A computation shows that
\[\one_{Z(\mu,\nu)}*\one_{Z(\beta, \gamma)} = \one_{Z(\mu,\nu)Z(\beta,\gamma)} \quad \text{and} \quad (\one_{Z(\mu,\nu)})^* = \one_{Z(\nu,\mu)}.\]
It is straightforward to check that
 \[\{\one_{Z(v)}, \one_{Z(e, r(e))}, \one_{Z(r(e), e)}\}\] is a Leavitt $E$-family in
$A_R(G_E)$, and the following is proved in Example~3.2 of \cite{CS2015}.
\begin{thm}\label{thm-model}
\label{E to G}
Let $E$ be a row-finite directed graph with no sources and let $R$ a commutative ring with identity.    Then 
$s_e\mapsto \one_{Z(e, r(e))}$ induces a $*$-algebra isomorphism of $L_R(E)$ onto $A_R(G_E)$ such that
\begin{equation}
\label{groupoid normal}
\sum_{(\mu,\nu)\in F} r_{\mu,\nu} s_{\mu}s_{\nu^*}\mapsto \sum_{(\mu,\nu)\in F} r_{\mu,\nu}\one_{Z(\mu,\nu)}.
\end{equation}
\end{thm}
The isomorphism maps the diagonal $D(E)$ of $L_R(E)$ onto the commutative *-subalgebra $\lsp_R\{\one_{Z(\mu)} :\mu\in P(E)\}$.

Let $\mu,\nu,\beta,\gamma\in P(E)$. In $L_R(E)$ we have the multiplication formula,
\[s_\mu s_{\nu^*}s_\beta s_{\gamma^*}=\begin{cases}  s_{\mu\beta'}s_{\gamma^*} &\text{if~}\exists \beta'\in P(E) \text{~such that~} \beta=\nu\beta'\\
s_\mu s_{(\gamma\nu')^*}&\text{if~}\exists \nu'\in P(E) \text{~such that~} \nu=\beta\nu'\\
0 &\text{otherwise}
\end{cases}\]
(see, for example, \cite[Corollary~1.14]{R-CBMS}).  This translates into 
\begin{equation}\label{it:bp1} Z(\mu,\nu)Z(\beta,\gamma)=
\begin{cases}  Z(\mu\beta',\gamma) &\text{if~}\exists \beta'\in P(E) \text{~such that~} \beta=\nu\beta'\\
Z(\mu,\gamma\nu') &\text{if~}\exists \nu'\in P(E) \text{~such that~} \nu=\beta\nu'\\
\emptyset &\text{otherwise}.
\end{cases}
\end{equation}

We will work exclusively with the Steinberg algebra model of $L_R(E)$. We  abuse notation and write $L_R(E)$  for $A_R(G_E)$ and $D(E)$ for the image of the diagonal $D(E)$  of $L_R(E)$ in $A_R(G)$.  In particular, we view elements of $L_R(E)$ as functions from $G_E$ to $R$.

%
\section{Normalizers in  the Leavitt path algebra.}\label{sec-normalizers}

Throughout this section $E$ is a row-finite graph with no sources and $R$ is an integral domain with identity.  We view the Leavitt path algebra $L_R(E)$ as the Steinberg algebra of the graph groupoid $G_E$ with diagonal 
\[
D(E):=\lsp\{\one_{Z(\mu)}: \mu \in P(E)\}.
\]

\begin{definition}
An $n\in L_R(E)$ is  a \emph{normalizer} of $D(E)$ if $ndn^*, n^*dn\in D(E)$ for all $d\in D(E)$.
We denote the set of normalizers of $D(E)$ by $N(E)$.
\end{definition}

In this section we prove that there is a partial action $\alpha$ of $N(E)$ on $E^\infty$, 
and we investigate the properties of the normalizers. This will allow us in \S\ref{secWeyl}
 to construct an algebraic analogue of the extended Weyl groupoid from \cite[Section~4]{BCW}.

Let $n, m\in N(E)$. Then $n^*$ and $nm\in N(E)$. Because $D(E)$ is a $*$-subring, $d\in N(E)$ for all $d\in D(E)$.  Since $D(E)$ contains the  set $\{\sum_{v\in V} \one_{Z(v)}: \text{$V$ is a finite subset of $E^0$}\}$ of local units for $L_R(E)$, we have $nn^*$ and $n^*n\in D(E)$.  In particular
\[
\supp(nn^*)\cup \supp(n^*n)\subset E^\infty.
\]
If $n=\sum_{(\mu,\nu)\in F} r_{\mu,\nu}\one_{Z(\mu,\nu)}$ is in normal form, then the $Z(\mu,\nu)$ are mutually disjoint, and hence 
the support of $n$ is the disjoint union of the $Z(\mu,\nu)$
and $\supp(n^*)=\supp(n)\inv$.

\begin{example}  Let $\mu,\nu\in P(E)$ with $s(\mu)=s(\nu)$.  We show that  $\one_{Z(\mu, \nu)} \in N(E)$. 
It suffices to see that   $\one_{Z(\mu, \nu)}\one_{Z(\tau)}\one_{Z(\nu, \mu)}$ is in  $D(E)$ for all $\tau\in P(E)$.  Suppose $0\neq \one_{Z(\mu, \nu)}\one_{Z(\tau)}\one_{Z(\nu, \mu)}$. Then either $\tau=\nu\tau'$ or $\nu=\tau\nu'$ for some $\tau', \nu'\in P(E)$, and then the product is either $Z(\mu\tau')$ or $Z(\mu)$. Thus $\one_{Z(\mu, \nu)}\in N(E)$.
\end{example}

The next proposition is the key to defining a partial action $\alpha$ of $N(E)$ on $E^\infty$.

\begin{prop}\label{prop:supp n} 
Let $n\in N(E)$ and let $x\in\supp(n)$.   Then 
\begin{enumerate}
\item \label{it:lemsupp 1} $\supp(n)\supp (n)\inv\cup \supp(n)\inv\supp(n)\subset \Iso(G_E):=\{g\in G_E: r(g)=s(g)\}$;
\item\label{it:lemsupp 2} $r(\supp(n)x)$ is a singleton.
\end{enumerate}
\end{prop}   

\begin{proof}
\eqref{it:lemsupp 1} Write $n=\sum_{(\mu,\nu)\in F} r_{\mu,\nu}\one_{Z(\mu,\nu)}$ in normal form. Since \[\supp(n)\inv\supp(n)=\supp(n^*)\supp(n^*)\inv,\] it suffices to show \[\supp(n)\supp (n)\inv\subset \Iso(G_E).\]   Aiming for a  contradiction, we suppose there exist $g_1,g_2\in \supp(n)$ such that $g_1g_2\inv\in \supp(n)\supp(n)\inv\setminus\Iso(G_E)$. Then  $s(g_1)=s(g_2)$ and $r(g_1)\neq r(g_2)$. Since the $Z(\mu,\nu)$ are mutually disjoint  and $s|_{Z(\mu,\nu)}$ is injective, there exist unique pairs $(\mu_i,\nu_i)\in F$ such that $g_i\in Z(\mu_i,\nu_i)$ and $Z(\mu_1,\nu_1)\neq Z(\mu_2,\nu_2)$.

Suppose there exists $(\mu,\nu)\in F$ such that $s(g_1)\in Z(\nu)$. Then $\nu=s(g_1)(0,|\nu|)=s(g_1)(0,|\nu_1|)=\nu_1$. In particular, since $s(g_1)=s(g_2)$ we have $\nu_2=\nu_1$.  Define 
\[
\Omega=\{\mu: (\mu,\nu_1)\in F\}.
\]
Since $r|_{Z(\mu,\nu)}$ is injective, for  each $(\mu,\nu)\in F$, $r(g_1) Z(\mu,\nu)$ contains at most one element.  Thus  $r(g_1) \supp(n)$ is finite and hence so is $s(r(g_1) \supp(n))$; it is nonempty since $s(g_1)\in s(r(g_1) \supp(n))$.   
Choose  $k>|\nu_1|$ such that
\[
y\in s(r(g_1) \supp(n))\setminus\{s(g_1)\}\quad\Longrightarrow \quad y\notin Z(s(g_1)(0,k)).
\]
Set $\tau=s(g_1)(0,k)$. Then  $r(g_1)\supp(n) Z(\tau)=\{s(g_1)\}$.

Of course,  $\one_{Z(\tau)}\in D(E)$ and $n$ is a normalizer, and so
\begin{equation}
\label{eq:conj}
n\one_{Z(\tau)} n^*= \sum_{\kappa,\mu\in \Omega} r_{\kappa, \nu_1}\overline{r_{\mu, \nu_1}} \one_{Z(\kappa, \nu_1)Z(\tau)Z(\nu_1, \mu)}\in D(E).
\end{equation}
Since $R$ is an integral domain and $\tau$ extends $\nu_1$, each $(\kappa,\mu)$-summand in \eqref{eq:conj} is non-zero.
By assumption,  $g_1g_2\inv \notin \Iso(G_E)$, and in particular $g_1g_2\inv \notin \go_E$.  Since $n\one_{Z(\tau)} n^*\in D(E)$ we must have
\[
n\one_{Z(\tau)} n^*(g_1g_2\inv)=0.
\]
But, for example, we have $\mu_1, \mu_2\in\Omega$ with $\one_{Z(\mu_1,\nu_1)Z(\tau) Z(\nu_1, \mu_2)}(g_1g_2\inv)\neq 0$, and hence  there must exist $\mu_3,\mu_4\in \Omega$ with $(\mu_1,\mu_2)\neq (\mu_3,\mu_4)$ such that 
\[
\one_{Z(\mu_3,\nu_1)Z(\tau) Z(\nu_1, \mu_4)}(g_1g_2\inv)\neq 0.
\]

For $i=1,2$, set
\[
F_i:=\{(\mu,\nu_1)\in F: r(g_i)\in Z(\mu)\}.\]
Note that $(\mu_1,\nu_1), (\mu_3,\nu_1)\in F_1$ and $(\mu_2,\nu_1),(\mu_4,\nu_1)\in F_2$.  Thus both $F_1$ and $F_2$ have at least 2 elements. 

We will obtain a contradiction by showing that $F_1$ and $F_2$ are both singletons.
For $i=1,2$, define
$c_i: F_i\to \Z$ by $c_i((\mu,\nu_1))=|\mu|-|\nu_1|$.
Then $c_i$ is injective because $c_i((\mu, \nu_1))=c_i((\mu', \nu_1))$ implies that $|\mu|=|\mu'|$ and so $\mu=r(g_i)(0,|\mu|)=r(g_i)(0,|\mu'|)=\mu'$.

Since $c_1 $ is injective, there exist unique elements  $(\kappa_{\max},\nu_1), (\kappa_{\min},\nu_1)\in F_1$ such that $|\kappa_{\max}|-|\nu_1|$ is a maximum and $|\kappa_{\min}|-|\nu_1|$ is a minimum. Similarly, there exist unique elements  $(\mu_{\max},\nu_1), (\mu_{\min},\nu_1)\in F_2$ such that $|\mu_{\max}|-|\nu_1|$ is a maximum and $|\mu_{\min}|-|\nu_1|$ is a minimum.  It follows that the
$(\kappa_{\max},\mu_{\min})$-summand of \eqref{eq:conj}  has degree $|\kappa_{\max}|-|\mu_{\min}|$, the largest possible. Similarly, the  $(\kappa_{\min},\mu_{\max})$-summand of \eqref{eq:conj}  has degree $|\kappa_{\min}|-|\mu_{\max}|$, the smallest possible. If either $|\kappa_{\max}|-|\mu_{\min}|\neq 0$ or $|\kappa_{\min}|-|\mu_{\max}|\neq 0$, then these summands cannot be cancelled, contradicting $n\one_{Z(\tau)} n^*(g_1g_2\inv)=0$. Thus $|\kappa_{\max}|=|\mu_{\min}|$ and $|\kappa_{\min}|=|\mu_{\max}|=0$.
Now $|\kappa_{\max}|=|\mu_{\min}|\leq|\mu_{\max}|=|\kappa_{\min}|$, and then we must have  equality throughout.  Since $c_1$ is injective, $|\kappa_{\max}|=|\kappa_{\min}|$ implies $\kappa_{\max}=\kappa_{\min}$. Similarly $\mu_{\max}=\mu_{\min}$. It follows that both $F_1$ and $F_2$ are singletons, a contradiction.  Thus $\supp(n)\supp (n)\inv\subset \Iso(G_E)$.  This gives  \eqref{it:lemsupp 1}.

\eqref{it:lemsupp 2} Since $x\in \supp(n)$ we have $r(\supp(n)x)\neq \emptyset$. Let $u_1$ and $u_2\in r(\supp(n)x)$.  Then there exist $g_1,g_2\in \supp(n)x$ with $r(g_i)=u_i$ for $i=1,2$.  Now \eqref{it:lemsupp 1} implies\[g_1g_2\inv\in \supp(n)\supp(n)\inv\subset \Iso(G).\] Thus $u_1=r(g_1)=r(g_1g_2\inv)=s(g_1g_2\inv)=r(g_2)=u_2$, 
and  $r(\supp(n)x)$ is  a singleton.
\end{proof}

\begin{lemma}\label{baby steps}
Let $n\in N(E)$. Define
\[
\dom(n):=
\supp(n^*n) \text{\ and\ }
\ran(n):=
\supp(nn^*).
\]
\begin{enumerate}
\item 
Let  $x\in s(\supp(n))$. Then $\alpha_n(x):=r(\supp(n)x)$  gives a well-defined function $\alpha_n:\supp(n)\to E^\infty$.
\item 
Let $d\in D(E)$ and view $d \circ \alpha_n:G_E \to R$  by setting $d \circ \alpha_n(g) = 0$ if
$g \notin s(\supp(n))$.  Then 
\begin{align}
n^*d&=(d\circ\alpha_n)n^*\label{eq:dn-new} \text{ and }\\
n^*dn&=(d\circ\alpha_n)n^*n\label{eq:dn2}.
\end{align}
\item
If $x \in \dom(n)$, then $\alpha_n(x)\in\ran(n)$.
\end{enumerate}
\end{lemma}

\begin{proof} 
Let  $x\in s(\supp(n))$. Then $\alpha_n(x):=r(\supp(n)x)$ is well-defined by Proposition~\ref{prop:supp n}\eqref{it:lemsupp 2}. Let $d\in D(E)$ and let $g\in G_E$.
Then 
\begin{align*}
(n^*d)(g)
&=\sum_{r(g)=r(h)}n^*(h)d(h^{-1}g)\\
&=n^*(g)d(s(g))
\end{align*}
and
\begin{align*}((d\circ\alpha_n)n^*)(g) &=  \sum_{r(g)=r(h)}d(\alpha_n(h))n^*(h^{-1}g)\\
&=n^*(g)d(\alpha_n(r(g)))\\ 
&= n^*(g)d(s(g))
\end{align*}
giving \eqref{eq:dn-new}.  Multiplying on the right by $n$ gives \eqref{eq:dn2}. Applying \eqref{eq:dn2} with $d=nn^*$  and evaluating at $x\in\dom(n)$ gives
\[
(n^*n(x))^2=n^*(nn^*)n(x)=nn^*(\alpha_n(x))n^*n(x).
\]
 Since $R$ is an integral domain, $x\in\dom(n)$ implies $(n^*n(x))^2\neq 0$, and then that $\alpha_n(x)\in\ran(n)$.
\end{proof}

\begin{example}\label{ex alpha n}
Consider the normalizer $\one_{Z(\mu,\nu)}$ of $D(E)$. Let $x\in \dom( \one_{Z(\mu,\nu)})=Z(\nu)$. Then $r(Z(\mu,\nu)x)=\mu\sigma^{|\nu|}(x)$. Thus 
\[\alpha_{\one_{Z(\mu,\nu)}}(x)=\mu\sigma^{|\nu|}(x).
\]
Now consider a normalizer $n=\sum_{(\mu,\nu)\in F} r_{\mu,\nu} \one_{Z(\mu,\nu)}$  in normal form. Let $x\in \dom(n)$. Then $x\in Z(\nu)$ for a unique $\nu_0$ and at least one pair $(\mu, \nu_0)\in F$.  By Proposition~\ref{prop:supp n}\eqref{it:lemsupp 2},  $r(\supp(n)x)$ is a singleton. Thus 
\[
\alpha_{n}(x)=\alpha_{\one_{Z(\mu,\nu_0)}}(x)
\]
for any $(\mu,\nu_0)\in F$.
\end{example}

We now prove that $\alpha$ is a partial action of the normalizers $N(E)$  on the infinite path space $E^\infty$. In the proof, we will use the following observation:
\begin{lemma}
\label{lem:range transfer}
Let  $G$ be a groupoid and $X, Y, Z\subset G$. Then
\[r(X r(YZ))=r(XYZ).
\]
\end{lemma}

\begin{prop}
\label{prop:alg act new}
Let $m,n\in N(E)$, and $d\in D(E)$.  Recall that for $x\in s(\supp(n))$, $\alpha_n(x):=r(\supp(n)x)$.
 \begin{enumerate}
 \item\label{it:algact new 1} We have $x\in \dom(mn)$ if and only if $x\in \dom(n)$ and  $\alpha_n(x)\in \dom(m)$, and then $\alpha_{m}\circ\alpha_n=\alpha_{mn}$.
 \item \label{it:algact new 3} If $x\in\supp(d)=\dom(d)$, then $\alpha_d(x)=x$.
 \item \label{it:algact new 4} The map $\alpha_n$ is a homeomorphism of $\dom(n)$ onto $\ran(n)=\dom(n^*)$ with $\alpha_n^{-1}=\alpha_{n^*}$.
 \end{enumerate}
 \end{prop}

 \begin{proof}\eqref{it:algact new 1} As in Lemma~\ref{baby steps}, we consider the extension of $m^*m\circ \alpha_n$ to all of $G_E$ by zero. Then for all $x\in G_E^0$,  \eqref{eq:dn2} gives 
 \begin{equation*}\label{lineupsupports}
 (mn)^*(mn)(x)=n^*(m^*m)n(x)=m^*m(\alpha_n(x))n^*n(x)
 \end{equation*}
as elements of  $R$.  Now let $x\in\dom (n)$ and $\alpha_n(x)\in \dom(m)$.  Then $m^*m(\alpha_n(x))\neq 0$ and $n^*n(x)\neq 0$. Since $R$ is an integral domain, their product is nonzero also, and hence   $(mn)^*(mn)(x)\neq 0$. Thus $x\in\dom(mn)$.  Conversely, let $x\in \dom(mn)$. Then $(mn)^*(mn)(x)\neq 0$, and it follows that both $m^*m(\alpha_n(x))$ and $n^*n(x)$ are nonzero. Thus $x\in\dom (n)$ and $\alpha_n(x)\in \dom(m)$.
 
Now assume that $x\in\dom(mn)$. Then $x\in \dom(n)$ and $\alpha_n(x)\in \dom(m)$. To show that $\alpha_{m}\circ\alpha_n=\alpha_{mn}$, we need to show
  \[
  r(\supp(m)(\supp(n)x))=r(\supp(mn) x).
  \]  
 For this we compute with 
 \[
  n=\sum_{(\mu,\nu)\in F} r_{\mu,\nu} \one_{Z(\mu,\nu)}\quad\text{and}\quad m=\sum_{(\mu',\nu')\in F'} r_{\mu',\nu'} \one_{Z(\mu',\nu')}
 \]
in normal form, and   \[mn= \sum_{(\mu',\nu')\in F',(\mu,\nu)\in F}r_{\mu',\nu'} r_{\mu,\nu} \one_{Z(\mu',\nu')Z(\mu,\nu)}.
  \]
There exists $(\mu_0,\nu_0)\in F, (\mu'_0,\nu'_0)\in F'$ such that $x\in Z(\nu_0)$, $\alpha_n(x)=r(Z(\mu_0,\nu_0)x)$, $\alpha_n(x)\in Z(\nu_0')$, and 
  \[
  \alpha_{m}(\alpha_n(x))=r(Z(\mu_0',\nu_0')\alpha_n(x))=r(Z(\mu_0',\nu_0')(r(Z(\mu_0,\nu_0)x))).
  \]
  But $r_{\mu_0',\nu_0'} r_{\mu_0,\nu_0} \one_{Z(\mu_0',\nu_0')Z(\mu_0,\nu_0)}$ is a nonzero summand  in $mn$ with $x\in s(Z(\mu_0',\nu_0')Z(\mu_0,\nu_0))$. Thus $\alpha_{mn}(x)=r(Z(\mu_0',\nu_0')Z(\mu_0,\nu_0)x)$ by the definition of $\alpha$.  By Lemma~\ref{lem:range transfer} we have
\[
  \alpha_{m}(\alpha_n(x))=r(Z(\mu_0',\nu_0')(r(Z(\mu_0,\nu_0)x)))=r(Z(\mu_0',\nu_0')Z(\mu_0,\nu_0)x)=\alpha_{mn}(x).
\]

\eqref{it:algact new 3} Let $d=\sum_{\mu\in F} r_\mu \one_{Z(\mu)}\in D(E)$ in normal form. Then $x\in\supp(d^*d)$ implies $x\in Z(\mu)$ for some $\mu\in F$. Now $\alpha_d(x)=\alpha_{\one_{Z(\mu,\mu)}}(x)=\mu\sigma^{|\mu|}(x)=x$ by Example~\ref{ex alpha n}.
  
\eqref{it:algact new 4} Since $nn^*$ and $n^*n\in D$, we have $\alpha_n\circ\alpha_{n^*}(x)=\alpha_{nn^*}(x)=x$ by \eqref{it:algact new 1} and \eqref{it:algact new 3}. Similarly, $\alpha_{n^*}\circ\alpha_{n}(x)=x$, and hence $\alpha_n:\dom(n)\to\ran(n)=\dom(n^*)$ is invertible with $\alpha_n^{-1}=\alpha_{n^*}$.  

By symmetry, it suffices to show that $\alpha_n$ is continuous (for then so is $\alpha_{n^*}=\alpha_{n}^{-1}$). Take $n$ in normal form as above.    Suppose that $x_k\to x$ in $\dom(n)$.  There exists $(\mu_0, \nu_0)\in F$ such that $x=\nu_0x'\in Z(\nu_0)$.  Eventually, $x_k\in Z(\nu_0)$, that is, $x_k=\nu_0x_k'$. Then $x_k'\to x'$. By Example~\ref{ex alpha n} we have
\[
\alpha_n(x_k)=\mu_0x_k'\to\mu_0x'=\alpha_n(x).
\]
Thus $\alpha_n$ is continuous.  It follows that $\alpha_n$ is a homeomorphism.
 \end{proof}

We will construct the algebraic analogue of the extended Weyl groupoid in \S\ref{secWeyl}. This construction  uses the partial action $\alpha_n$, and we will have to work separately with isolated points in $E^\infty$.  If $x\in E^\infty$ is isolated, then there exists $k\in\N$ such that $Z(x(0,k))=\{x\}$ (and then $Z(x(0,l))=\{x\}$ for all $l\geq k$).  We write $p_x$ for $\one_{Z(x(0,k))}= \one_{\{x\}}$.  For example, if  $x$ satisfies $r^{-1}(r(x_i))=\{x_i\}$ for eventually all $i$, then $x$ is isolated in $E^\infty$. In particular, if $\eta$ is a cycle without entry, then $\tau\eta\eta\eta...$ is isolated for any $\tau\in P(E)$ with $s(\tau)=r(\eta)$.

\begin{lemma}
\label{lem:isolated}
Let $x\in E^\infty$ be an isolated path and let $n\in N(E)$. 
\begin{enumerate}
\item \label{it:iso1} If $p_xnp_x$ is nonzero, then $p_xnp_x = r\one_{\{(x,k,x)\}}$ for some $k\in\Z$ and $r\in R$.
\item \label{it:iso3} $p_xnp_x\in N(E)$.
\item \label{it:iso4} $p_xnp_x$ is  homogeneous. If $x=\tau\eta\eta\cdots$ where $\eta$ is a cycle without entry,
 then the degree of $p_xnp_x$  is a multiple of the length of $\eta$; otherwise the degree is $0$.
\end{enumerate}
\end{lemma}

\begin{proof}

For \eqref{it:iso1}, write
\[
n=\sum_{(\mu,\nu)\in F} r_{\mu,\nu} \one_{Z(\mu,\nu)}
\]
in normal form.  Choose  $M\geq \max\{|\mu|, |\nu|:(\mu,\nu)\in F\}$ such that  $Z(x(0,M))=\{x\}$.  Then $p_x=\one_{Z(x(0,M))}$ and
\[
p_xnp_x=\sum_{(\mu,\nu)\in F} r_{\mu,\nu} \one_{Z(x(0,M))Z(\mu,\nu)Z(x(0,M))}.
\]
By \eqref{it:bp1},  a summand
$r_{\mu,\nu} \one_{Z(x(0,M))Z(\mu,\nu)Z(x(0,M))}\neq 0$ if and only if $\mu=x(0,|\mu|)$ and $\nu=x(0,|\nu|)$. But all the $\nu$ have the same length $l$, and with $Q=\set{\mu: (\mu, x(0,l))\in F, \mu=x(0,|\mu|)}$ we have 
\begin{equation}
\label{eq:pnp}
p_xnp_x=\sum_{\mu\in Q} r_{\mu,x(0,l)} \one_{Z(x(0,M))Z(\mu,x(0,l))Z(x(0,M))}.
\end{equation}
If $\mu,\mu'\in Q$  with $|\mu|=|\mu'|$ then $\mu=\mu'$.  Since $Q$ is finite there exists a unique $\mu_{\min}$ of smallest length and a unique $\mu_{\max}$ of largest length. Since $n$ is a normalizer $p_xnp_xn^*p_x$ is in the diagonal, and is therefore homogeneous of degree zero.  Now
\begin{equation}
p_xnp_xn^*p_x
=\sum_{\mu,\mu'\in Q} r_{\mu,x(0,l)}\overline{r_{\mu',x(0,l)}} \one_{Z(x(0,M))Z(\mu,x(0,l))Z(x(0,M))Z(x(0,M),\mu')Z(x(0,M))}.\label{eq:diagele}
\end{equation}
Consider the summand \[ r_{\mu_{\max},x(0,l)}\overline{r_{\mu_{\min},x(0,l)}} \one_{Z(x(0,M))Z(\mu_{\max},x(0,l))Z(x(0,M))Z(x(0,M),\mu_{\min})Z(x(0,M))}\] in the above sum: it is nonzero since $R$ is an integral ideal domain and \[
Z(x(0,M))Z(\mu_{\max},x(0,l))Z(x(0,M))Z(x(0,M),\mu_{\min})Z(x(0,M))\neq \emptyset,\]  and it  has degree $|\mu_{\max}|-|\mu_{\min}|$.  By definition of $\mu_{\max}$ and $\mu_{\min}$, it is the unique summand in \eqref{eq:diagele} of maximal degree.  So if $|\mu_{\max}|-|\mu_{\min}|\neq 0$, then there is no other summand in \eqref{eq:diagele} to cancel with it. Thus   $\mu_{\max}=\mu_{\min}$, and now $Q$ is a singleton. Now we apply \eqref{it:bp1} to the
single summand of (\ref{eq:pnp}) to get  $p_xnp_x=r\one_{Z(\beta, \gamma)}$ for some $\beta, \gamma \in P(E)$ and $r\in R$.  Now 
\[Z(\beta, \gamma)=\supp(p_xnp_x) \subset \{x\}\supp(n)\{x\} \subset \{x\}G_E\{x\}.\] 
Since $r$ and $s$ restricted to $Z(\beta, \gamma)$ are injective,
 $Z(\beta, \gamma)= \{(x,|\beta|-|\gamma|, x)\}$.  This gives \eqref{it:iso1} with $k= |\beta|-|\gamma|$.

For \eqref{it:iso3}, $p_xnp_x=r\one_{Z(\beta, \gamma)}$ is a normalizer  by Example~\ref{ex alpha n}.
 
For \eqref{it:iso4}, $p_xnp_x=r\one_{Z(\beta, \gamma)}$ is homogeneous of degree $|\beta|-|\gamma|$. Since $(x,|\beta|-|\gamma|, x)\in G_E$ we have 
$\sigma^{|\beta|}(x) = \sigma^{|\gamma|}(x)$.
If  $x=\tau\eta\eta\cdots$ where $\eta$ is a cycle without entry, then $|\beta|-|\gamma|$ must be a multiple of the length of $\eta$.
Otherwise, $x$ is not periodic, and then $|\beta|-|\gamma|=0$.
\end{proof}

The  following lemma is the analogue  of Lemma~4.2 in \cite{BCW}.

\begin{lemma}
\label{lem:alphapx new}
 Let $n\in N(E)$ and let $x\in \dom(n)$ be an isolated path.
Then
\begin{enumerate}
\item\label{it:px1} $np_xn^*=nn^*p_{\alpha_n(x)}$ and 
\item\label{it:px3} $np_x=p_{\alpha_n(x)}n$.
\end{enumerate}
\end{lemma}
\begin{proof}
For \eqref{it:px1}, let $y\in\dom(n^*)$ and use \eqref{eq:dn2} with $d=p_x$ to get
\begin{align*}
np_xn^*(y)
&=p_x(\alpha_{n^*}(y))nn^*(y)\\
&=\begin{cases}nn^*(y)&\text{if $x=\alpha_n(y)$}\\
0&\text{else}
\end{cases}\\
&=p_{\alpha_n(x)}(y)nn^*(y)\\
&=[p_{\alpha_n(x)}nn^*](y).
\end{align*}
Thus $np_xn^*=nn^*p_{\alpha_n(x)}$. 
For \eqref{it:px3}, we compute the convolutions to see that
\begin{gather*}np_x(g)=\sum_{r(h)=r(g)}n(h)p_x(h^{-1}g)=n(g)p_x(s(g))
\\
p_{\alpha_n(x)}n(g)=\sum_{r(h)=r(g)}p_{\alpha_n(x)}(h)n(h^{-1}g)=n(g)p_{\alpha_n(x)}(r(g)).
\end{gather*}
Consider $g\in\supp(n)$. If $x=s(g)$, then $r(g)=r(\supp(n)x)$, that is, $r(g)=\alpha_n(x)$. On the other hand, if $r(g)=\alpha_n(x)$, then $r(g)=r(\supp(n)x)$ by definition of $\alpha_n$, and hence $x=s(g)$. Thus $np_x(g)=p_{\alpha_n(x)}n(g)$, and \eqref{it:px3} follows.
\end{proof}

\begin{lemma}\label{jackpot} 
Let $a\in L_R(E)$ be $0$-graded and let $x$ be an isolated path in $E^\infty$. Then $p_xap_x=rp_x$ for some $r\in R$.
\end{lemma}
\begin{proof}
Write $a= \sum_{(\mu,\nu)\in F}  r_{\mu,\nu}\one_{Z(\mu, \nu)}$ in normal form.
Since $a$ is zero graded, for each $(\mu,\nu) \in F$ we have $|\mu|=|\nu|$ (see \cite[Lemma~3.5]{CFST}).  
Then
\[
 p_xap_x = \sum_{(\mu,\nu)\in F}  r_{\mu,\nu}\one_{\{x\}Z(\mu, \nu)\{x\}}.\]
Because the elements of $F$ are mutually disjoint, there is at most one $(\mu, \nu) \in F$ such that
$(x,0,x) \in Z(\mu,\nu)$.  Thus either
$p_xap_x = r_{\mu,\nu}p_x$ or $p_xap_x=0=0p_x$.
\end{proof}


\section{The algebraic Weyl Groupoid}\label{secWeyl}
Throughout this section $E$ is a row-finite graph with no sources and $R$ is an integral domain with identity. 
We use the partial action $\alpha$ from \S\ref{sec-normalizers} to construct a groupoid associated to the inclusion of the diagonal $D(E)$ in  $L_R(E)$.  This gives an algebraic version of the ``extended Weyl groupoid'' of \cite{BCW}.

For the next definition, recall that if  $x\in E^\infty$ is isolated  and $n\in N(E)$ then  $p_xnp_x$ is homogeneous by Lemma~\ref{lem:isolated}.  

\begin{definition}
\label{def:eqrel}
We define a relation $\sim$ on 
\[\{(n,x): n\in N(E), x\in \dom(n)\}\]
by  $(n,x)\sim (m,x')$ if $x=x'$ and either
\begin{enumerate}
\item \label{eqcondiso} $x$ is an isolated path in $E^\infty$, $\alpha_n(x)=\alpha_{m}(x)$ and the degree of $p_xn^*mp_x$ is $0$; or
\item \label{eqcondnoniso} $x$ is not an  isolated path in $E^\infty$ and there exists  an open neighbourhood $V\subset \dom(n)\cap \dom(m)$ of $x$ such that $\alpha_n(y)=\alpha_{m}(y)$ for all $y\in V$.
\end{enumerate}
\end{definition}

\begin{remark}\label{rem:eqrel}
Item~\eqref{eqcondnoniso} says that $\alpha_n$ and $\alpha_{m}$ agree as germs: this is exactly the condition used in \cite[Definition~4.2]{Ren} and is  condition~(b) in \cite[Proposition~4.6]{BCW}.  Condition~\eqref{eqcondiso} is {\em a priori } different from condition~(a) of  \cite[Proposition~4.6]{BCW}. Condition~(a) requires the $K_1$ class of a  unitary $U_{x, n, m}:= cp_xn^*mp_x$ (for the constant $c$ which makes $U$ a unitary)  to be zero.  However,  $p_x C^*(E) p_x$ is Morita equivalent to either $C(\mathbb{T})$ or $\C$ and $K_1(C(\mathbb{T}))\cong \Z$ by an isomorphism sending $z^n$ to $n$.  A computation now shows the $K_1$ class of the unitary $U_{x,n,m}$ is precisely the degree of $p_xn^*mp_x$.
\end{remark}

\begin{lemma}\label{lem-equivalence relation} 
The relation in Definition~\ref{def:eqrel} is an equivalence relation.
\end{lemma}
\begin{proof}  That $\sim$ is an equivalence relation when restricted to  \[\{(n, x) : n \in  N(E), x \in \dom(n)\text{\ is not isolated}\}\] follows from the proof of \cite[Proposition~4.6]{BCW}. 

Let $x$ be an isolated path in $E^\infty$.  Everything is straightforward, except perhaps transitivity. So suppose that $(n_1, x)\sim (n_2,x)$ and $(n_2, x)\sim (n_3,x)$. Then $\alpha_{n_1}(x)=\alpha_{n_2}(x)$ and $p_xn_1^*n_2p_x$ is $0$-graded, and $\alpha_{n_2}(x)=\alpha_{n_3}(x)$ and $p_xn_2^*n_3p_x$ is $0$-graded. Thus $\alpha_{n_1}(x)=\alpha_{n_3}(x)$.  We compute, using Lemma~\ref{lem:alphapx new}\eqref{it:px1} and then \eqref{eq:dn-new}:
\begin{align}
(p_xn_1^*n_2p_x)(p_xn_2^*n_3p_x)
&=p_xn_1^*(n_2p_xn_2^*)n_3p_x\notag\\
&=p_xn_1^*(n_2n_2^*p_{\alpha_{n_2}(x)})n_3p_x \notag\\
&=(n_2n_2^*p_{\alpha_{n_2}(x)}\circ\alpha_{n_2n_2^*p_{\alpha_{n_2}(x)}})p_xn_1^*n_3p_x.\label{new-equation-6}
\end{align} 
Since $p_xn_1^*n_2p_x$ and $p_xn_2^*n_3p_x$ are both $0$-graded so is their product. Thus \eqref{new-equation-6} is $0$-graded. Since $n_2n_2^*p_{\alpha_{n_2}(x)}\circ\alpha_{n_2n_2^*p_{\alpha_{n_2}(x)}}\in D(E)$, it follows that $p_xn_1^*n_3p_x$ is $0$-graded as well. We have proved $(n_1, x)\sim (n_3,x)$.
\end{proof}

\begin{lemma}
\label{lem:eqrelond}
Let  $d, d'\in D(E)$ and let $x\in \supp(d)\cap \supp(d')$.
Then $(d,x)\sim (d',x)$.
\end{lemma}

\begin{proof} First suppose that $x$ is not isolated. 
By Proposition~\ref{prop:alg act new}(\ref{it:algact new 3}),  $\alpha_d(y) =y= \alpha_{d'}(y)$ for all 
$y$ in the open neighbourhood $\supp(d)\cap \supp(d')$ of $x$. Thus $(d,x)\sim (d',x)$.

Second, suppose that $x$ is isolated. We still have $\alpha_d(x) = x=\alpha_{d'}(x)$.
Now $\supp (p_xd^*d'p_x) \subset \go$ implies that  $p_xd^*d'p_x$ has degree $0$.    
Thus again $(d,x)\sim (d',x)$.
\end{proof}

\begin{prop}\label{prop:Weyl}
Let $\sim$ be the equivalence relation on $\{(n, x) : n \in  N(E), x \in \dom(n)\}$ from Lemma~\ref{lem-equivalence relation}. Denote the collection of equivalence classes by $W_E$. 
Then $W_E$, equipped with the following structure, is a groupoid:
\begin{itemize}
\item partially-defined product $[(n_1, x_1)][(n_2, x_2)] := [(n_1n_2, x_2)]$  if $\alpha_{n_2}(x_2) = x_1$;
\item unit space $W_E^{(0)}:=\{[(d,x)]:d\in D(E), x\in \dom(d)\}$;
\item range and source maps $r,s:W_E\to W_E^{(0)}$ defined by
\[
r([(n,x)]):=[(nn^*,\alpha_n(x))]\text{\ and \ }s([(n,x)]):=[(n^*n,x)];
\]
\item inverse $[(n,x)]^{-1}:=[n^*,\alpha_n(x)]$.
\end{itemize}
\end{prop}

\begin{proof} A groupoid is a small category with inverses; here the set of morphisms will be $W_E$ and the set of objects will be $W_E^{(0)}=\{[(d,x)]:d\in D, x\in \dom(d)\}$ as a subset of the morphisms. 

We start by showing that  the product  is well-defined. 
Suppose that  $(m_1,x_1)\sim (n_1, x_1)$, $(m_2,x_2)\sim (n_2, x_2)$, and $[(n_1, x_1)]$ and $[(n_2, x_2)]$ are composable, that is, $\alpha_{n_2}(x_2)=x_1$.
We consider two cases: either $x_1$ is an isolated path or it is not.  Since $\alpha_{n_2}$ is a homeomorphism, it follows that $x_1$ is isolated if and only if $x_2$ is.  In either case, we need to show that 
$[m_1,x_1]$ and $[m_2,x_2]$ are composable and $(n_1n_2, x_2)\sim (m_1m_2, x_2)$.

First suppose that $x_1$ is isolated. We know $\alpha_{n_1}(x_1)=\alpha_{m_1}(x_1)$, $\alpha_{n_2}(x_2)=\alpha_{m_2}(x_2)$, and $p_{x_1}n_1^*m_1p_{x_1}$ and $p_{x_2}n_2^*m_2p_{x_2}$ are $0$-graded. Thus $x_1=\alpha_{n_2}(x_2)=\alpha_{m_2}(x_2)$, and hence $[(m_1, x_1)]$ and $[(m_2,x_2)]$ are composable as well. We have
\[
\alpha_{n_1n_2}(x_2)=\alpha_{n_1}\circ\alpha_{n_2}(x_2)=\alpha_{n_1}(x_1)=\alpha_{m_1}(x_1)=\alpha_{m_1}\circ\alpha_{m_2}(x_2)=\alpha_{m_1m_2}(x_2).
\] 
We compute using Lemma~\ref{lem:alphapx new}\eqref{it:px3} several times:
\begin{align*}
p_{x_2}(n_1n_2)^*(m_1m_2)p_{x_2}
&=
(n_2p_{x_2})^*n_1^*m_1p_{\alpha_{m_2}(x_2)}m_2\\
&=(p_{\alpha_{n_2}(x_2)}n_2)^*n_1^*m_1p_{\alpha_{m_2}(x_2)}m_2\\
&=n_2^*(p_{x_1}n_1^*m_1p_{x_1})m_2\\
&=n_2^*(rp_{x_1})m_2
\end{align*}
for some $r\in R$ by Lemma~\ref{jackpot}. But we know
\[
p_{x_2}n_2^*m_2p_{x_2}=n_2^*p_{\alpha_{n_2}(x_2)}p_{\alpha_{m_2}(x_2)}m_2=n_2^*p_{x_1}m_2
\]
is $0$-graded, and now $n_2^*(rp_{x_1})m_2$ must be as well. Thus $p_{x_2}(n_1n_2)^*(m_1m_2)p_{x_2}$ is $0$-graded as required.  We have proved that the product is well-defined  when $x_1$ is isolated.

Second, suppose $x_1$ is not isolated. Then $x_2$ is not isolated either.  By definition of $\sim$ there exists an open neighbourhood $V$ of $x_1$  such that $V\subset \dom(n_1)\cap \dom(m_1)$ and $\alpha_{n_1}(v)=\alpha_{m_1}(v)$ for all $v\in V$. Similarly, there exists an open neighbourhood $Y$ of $x_2$  such that $Y\subset\dom(n_2)\cap \dom(m_2)$ and $\alpha_{n_2}(y)=\alpha_{m_2}(y)$ for all $y\in Y$. We have $x_1=\alpha_{n_2}(x_2)=\alpha_{m_2}(x_2)$, and hence $[(m_1, x_1)]$ and $[(m_2,x_2)]$ are composable as well.
 Let $U=\alpha_{n_2}^{-1}(V)\cap Y$.  Then $\alpha_{n_2}(x_2)=x_1$ implies $x_2\in U$. 
Since $\alpha_{n_2}^{-1}$ is a homeomorphism, $U$ is an open neighbourhood of $x_2$. Let $x\in U$. Then $x\in Y$ implies $x\in \dom(n_2)\cap\dom(m_2)$.  Also, $x\in \alpha_{n_2}^{-1}(V)$ implies $\alpha_{n_2}(x)\in \dom(n_1)\cap\dom(n_2)$. Now $x\in \dom(n_1n_2)\cap\dom(m_1m_2)$ by Lemma~\ref{prop:alg act new}.  Thus $U \subset \dom(n_1n_2)\cap\dom(m_1m_2)$. Finally,  
\[
\alpha_{n_1n_2}(x)=\alpha_{n_1}\circ \alpha_{n_2}(x)=\alpha_{n_1}\circ \alpha_{m_2}(x)=\alpha_{m_1}\circ \alpha_{m_2}(x)=\alpha_{m_1m_2}(x).
\]
Thus the product is also well-defined when $x_1$ is not isolated.  

It is now straightforward to observe that the product is associative: suppose that $x_1=\alpha_{n_2}(x_2)$ and $x_2=\alpha_{n_3}(x_3)$. Then
\begin{align*}
\big([(n_1, x_1)][(n_2, x_2)]\big)[(n_3, x_3)]  &=[(n_1n_2, x_2)][(n_3, x_3)]=[((n_1n_2)n_3, x_3)]\\
\intertext{and $\alpha_{n_2n_3}(x_3)=\alpha_{n_2}(x_2)=x_1$, whence}
[(n_1, x_1)]\big([(n_2, x_2)][(n_3, x_3)]\big)&=[(n_1, x_1)][(n_2n_3, x_3)]=[(n_1(n_2n_3), x_3)].
\end{align*}

To see that the range map is well-defined, suppose that $(n, x)\sim (m,x)$. Then $\alpha_n(x)=\alpha_m(x)$. 
Now because $nn^*, mm^* \in D(E)$,  we have 
$(nn^*,\alpha_n(x))\sim (mm^*,\alpha_m(x))$ by Lemma~\ref{lem:eqrelond}. 

We claim that $r([(n,x)])[(n,x)]=[(nn^*n,x)]$ is just $[(n,x)]$. If $x$ is isolated, this follows because $\alpha_{nn^*n}(x)=\alpha_n\circ \alpha_{n^*n}(x)=\alpha_n(x)$ and $p_xn^*nn^*np_x$ is $0$-graded. If $x$ is not isolated, consider the open neighbourhood $U=\dom(nn^*n)\cap\dom (n)$ of $x$: we have $y\in U$ implies that $\alpha_{nn^*n}(y)=\alpha_n\circ \alpha_{n^*n}(y)=\alpha_n(y)$ as required. In either case $(nn^*n,x)\sim (n,x)$, and hence $r([(n,x)])[(n,x)]=[(n,x)]$.
Similarly the source map is well-defined and $[(n,x)]s([(n,x)])=[(n,x)]$. 

Finally, we observe that 
\begin{gather*}
[(n,x)][(n^*,\alpha_n(x))]=[(nn^*,\alpha_n(x))]=r([(n,x)])\\
[(n^*, \alpha_n(x))][(n,x)]=[(n^*n,x)]=s([(n,x)]).\end{gather*} It now follows from the associativity of the product that $[(n,x)]^{-1}=[(n^*,\alpha_n(x))]$ is a well-defined inverse for $[(n,x)]$.
\end{proof}

For $n\in N(E)$ define
\[
Z(n):=\{[(n,x)]: x\in \dom(n)\}.
\]
It is easy to check that $\{Z(n):n\in N(E)\}$ is a basis for a topology on $W_E$.  It is not immediately clear that  $W_E$ equipped with this topology is a topological groupoid. We will deduce this from Theorem~\ref{thm:groupoid iso} below which shows that $G_E$ and $W_E$ are isomorphic as topological groupoids.  The remainder of this section
is devoted to proving this.

\begin{lemma}
\label{lem:themap} 
The map $\Phi=\Phi_E: G_E\to W_E$ defined by
\[
(\mu x, |\mu|-|\nu|, \nu x)\mapsto [(\one_{Z(\mu,\nu)}, \nu x)].
\]
is a well-defined groupoid homomorphism.
\end{lemma}

\begin{proof}   
We start by showing that $\Phi$ is well-defined. Suppose that $(\mu x, |\mu|-|\nu|, \nu x)=(\beta y, |\beta|-|\gamma|, \gamma y)$ in $G_E$ for some $\mu, \nu, \beta,\gamma\in P(E)$ and $x,y\in E^\infty$. We need to show that $(\one_{Z(\mu,\nu)}, \nu x)\sim (\one_{Z(\beta,\gamma)}, \gamma y)$.  We have $\mu x=\beta y$, $|\mu|-|\nu|=|\beta|-|\gamma|$ and $\nu x=\gamma y$.  First suppose that $\nu x$ is isolated.  We have
\[
\alpha_{\one_{Z(\mu,\nu)}}(\nu x)=\mu x=\beta y=\alpha_{\one_{Z(\beta,\gamma)}}(\gamma y)=\alpha_{\one_{Z(\beta,\gamma)}}(\nu x)
\]
and $p_{\nu x}\one_{Z(\nu,\mu)}\one_{Z(\beta,\gamma)}p_{\nu x}$ has degree $|\nu|-|\mu|+|\beta|-|\gamma|=0$. Thus $(\one_{Z(\mu,\nu)}, \nu x)\sim (\one_{Z(\beta,\gamma)}, \gamma y)$ when $\nu x$ is isolated.  Second, suppose that $\nu x$ is not isolated.  
Since $\mu x=\beta y$, $|\mu|-|\nu|=|\beta|-|\gamma|$ and $\nu x=\gamma y$ either $\mu$ extends $\beta$ \emph{and} $\nu$ extends $\gamma$, or  $\beta$ extends $\mu$ \emph{and} $\gamma$ extends $\nu$. Suppose $\mu$ extends $\beta$ and $\nu$ extends $\gamma$ (the other case is similar).
Let $z\in \dom(\one_{Z(\mu,\nu)})\cap\dom(\one_{Z(\beta,\gamma)})=Z(\nu)\cap Z(\gamma)=Z(\nu)$, which is a neighbourhood  of $\nu x$. Then 
\[
\alpha_{\one_{Z(\mu,\nu)}}(z)=r(Z(\mu, \nu)z)=r(Z(\gamma, \beta)z)=\alpha_{\one_{Z(\beta,\gamma)}}(z).
\]
  Thus $(\one_{Z(\mu,\nu)}, \nu x)\sim (\one_{Z(\beta,\gamma)}, \gamma y)$ when $\nu x$ is not isolated. We have shown that $\Phi$ is well-defined.

Next we show that $\Phi$ is a groupoid homomorphism. Suppose that
$(\mu x,|\mu|-|\nu|, \nu x)$ and $(\beta y, |\beta|-|\gamma|, \gamma y)$ are composable in $G_E$, that is $\nu x=\beta y$. Then
\begin{align*}
\Phi((\mu x,|\mu|-|\nu|, \nu x)(\beta y, |\beta|-|\gamma|, \gamma y))&=\Phi((\mu x,|\mu|+|\beta|-(|\nu|+|\gamma|), \gamma y))\\&=[(\one_{Z(\mu x(0,|\beta|),\gamma y(0,|\nu|))}, \gamma y)];\\
\Phi((\mu x,|\mu|-|\nu|, \nu x))\Phi((\beta y, |\beta|-|\gamma|, \gamma y))&=[\one_{Z(\mu,\nu)},\nu x][\one_{Z(\beta,\gamma)},\gamma y]\\
&=[(\one_{Z(\mu,\nu)}\one_{Z(\beta,\gamma)}, \gamma y)].
\end{align*}
Let $m:=\one_{Z(\mu x(0,|\beta|),\gamma y(0,|\nu|))}$ and $n:=\one_{Z(\mu,\nu)}\one_{Z(\beta,\gamma)}$. 
 We need to show that \[(m, \gamma y)\sim (n, \gamma y).\]  First suppose that $\gamma y$ is isolated. Since $x=\sigma^{|\nu|}(\beta y)$ we have $\sigma^{|\beta|}(x)=\sigma^{|\nu|}(y)$. Now
\begin{align*}
\alpha_m(\gamma y)
&=\mu x(0,|\beta|)y(|\nu|, \infty)=\mu x(0,|\beta|)\sigma^{|\nu|}(y)\\
&=
\mu x(0,|\beta|)\sigma^{|\beta|}(x)=\mu x=\alpha_{\one_{Z(\mu,\nu)}}(\nu x)\\
&=\alpha_{\one_{Z(\mu,\nu)}}(\beta y)=\alpha_{\one_{Z(\mu,\nu)}\one_{Z(\beta,\gamma)}}(\gamma y)=\alpha_n(\gamma y).
\end{align*}
The degree of 
\[
p_{\gamma y}m^*np_{\gamma y}=p_{\gamma y} ( \one_{Z(\mu x(0,|\beta|),\gamma y(0,|\nu|))})^* \one_{Z(\mu,\nu)}\one_{Z(\beta,\gamma)} p_{\gamma y}
\]
is $-(|\mu|+|\beta|-(|\nu|+|\gamma|))+(|\mu|-|\nu|+|\beta|-|\gamma|)=0$. Thus $(m, \gamma y)\sim (n, \gamma y)$ when $\gamma y$ is isolated.  
 
Second, suppose that $\gamma y$ is not isolated. Since $\nu x=\beta y$, either $\nu$ extends $\beta$ or vice versa. Assume $\nu=\beta \nu'$ (the other case is similar). Then $n=\one_{Z(\mu,\gamma\nu')}$ and $\dom (n)=Z(\gamma\nu')$. Since $\dom(m)=Z(\gamma y(0, |\nu|))$ we have $\dom (m)\cap \dom (n)=Z(\gamma y(0, |\nu|))$ is an open neighbourhood of $\gamma y$. Let $z\in Z(\gamma y(0, |\nu|))$. Then 
\[
x(0, |\beta|)=(\nu x)(|\nu|, |\nu|+|\beta|)=(\beta y)(|\nu|, |\nu|+|\beta|)=y(|\nu|-|\beta|,|\nu|)=y(|\nu'|, |\nu|).
\]
Now 
\begin{align*}
\alpha_m(z)&=\alpha_{\one_{Z(\mu x(0,|\beta|),\gamma y(0,|\nu|))}}(z)=\mu x(0,|\beta|) \sigma^{|\gamma|+|\nu|}(z)\\
&=\mu y(|\nu'|, |\nu|) \sigma^{|\gamma|+|\nu|}(z)
=\mu z(|\gamma|+|\nu'|, |\gamma|+|\nu|)\sigma^{|\gamma|+|\nu|}(z)\\
&=\mu \sigma^{|\gamma|+|\nu|}(z)
=\alpha_{\one_{Z(\mu,\gamma\nu')}}(z)=\alpha_n(z).
\end{align*}
Thus $(m, \gamma y)\sim (n, \gamma y)$ when $\gamma y$ is not  isolated.  
We have shown that $\Phi$ is a groupoid homomorphism.
\end{proof}

To see that $\Phi$, as  defined in Lemma~\ref{lem:themap},  is injective we need another lemma.

\begin{lemma}
 \label{lem:getiandj}
Let $x \in E^{\infty}$, $\mu,\nu \in P(E)$ and $\eta$ a cycle in $P(E)$ such that
\[\nu x = \kappa \eta \eta \eta\dots  \quad \text{ and }\quad  \mu x = \lambda \eta\eta\eta\dots\]  where 
$\kappa$ and $\lambda$ in $P(E)$ are of minimal length (in the sense that $\kappa$ and $\lambda$ contain minimal subpaths of $\eta$). 
Then there exist $i,j\in\N$ such that $|\mu|-|\nu|=|\lambda\eta^i|-|\kappa\eta^j|$. 
\end{lemma}
  
\begin{proof}
There are two cases.  First, suppose $\kappa$ extends $\nu$, that is, $\kappa = \nu\nu'$.  Then 
\[\lambda \eta \eta \eta\dots =\mu x = \mu \sigma^{|\nu|}(\kappa \eta \eta\dots) = \mu \nu' \eta \eta \eta\dots\]
Since $\lambda$ is of minimal
length we must have $\mu\nu' = \lambda\eta^i$ for some $i\in \N$.
Thus 
\[|\mu|-|\nu|= |\mu \nu'| - |\nu \nu'|= |\lambda\eta^i|-|\kappa|= |\lambda\eta^i|-|\kappa\eta^0|\]
as needed. 

Second, suppose $\nu$ extends $\kappa$.  Then  $\nu = \kappa\eta^j\delta$ for some $j \in \N$ and some initial segment $\delta$  of $\eta$. Since $\nu x = \kappa \eta \eta \eta\dots$, we have
\[x=\delta'\eta\eta\eta\dots  \text{ where } \delta\delta'=\eta.
\]
Thus $\lambda \eta \eta \eta\dots  = \mu x = \mu \delta'\eta\eta\eta\dots$.  Because 
$\lambda$ has minimal length,  
there exists $i\in \N$ such that 
$|\mu \delta'| = |\lambda\eta^i|$.
Now
\[|\mu|-|\nu|= |\mu| - |\kappa\eta^j\delta|
=|\mu\delta'|- |\kappa\eta^j\delta\delta'|= |\lambda\eta^i|-|\kappa\eta^{j+1}|\]
as needed.
\end{proof}

\begin{lemma}
\label{lem:bij} The map $\Phi:G_E \to W_E$ defined  
in Lemma~\ref{lem:themap} is a bijection.
 
\end{lemma}
\begin{proof}
To see that $\Phi$ is injective,   suppose that $\Phi(\mu x,|\mu|-|\nu|,\nu x)=\Phi(\beta y,|\beta|-|\gamma|,\gamma y)$, that is, 
$(\one_{Z(\mu,\nu)}, \nu x)\sim (\one_{Z(\beta,\gamma)}, \gamma y)$.  Thus $\nu x=\gamma y$. 
First, suppose that $\nu x =\gamma y$ is isolated. Then
\[\mu x=\alpha_{\one_{Z(\mu,\nu)}}(\nu x)=\alpha_{\one_{Z(\beta,\gamma)}}(\nu x)=\alpha_{\one_{Z(\beta,\gamma)}}(\gamma y)=\beta y.\]
The degree of $p_{\nu x}\one_{Z(\nu,\mu)}\one_{Z(\beta,\gamma)}p_{\nu x}$ is $0$, that is, $|\nu|-|\mu|+|\beta|-|\gamma|=0$. Thus $(\mu x,|\mu|-|\nu|,\nu x)=(\beta y,|\beta|-|\gamma|,\gamma y)$ as required.  

Second, suppose that $\nu x=\gamma y$ is not isolated. Then there exists a neighbourhood $V\subset \dom(\one_{Z(\mu,\nu)})\cap\dom(\one_{Z(\beta,\gamma)})$ of $\nu x$ such that $\alpha_{\one_{Z(\mu,\nu)}}(v)=\alpha_{\one_{Z(\beta,\gamma)}}(v)$ for all $v\in V$.  Thus $\mu x=\alpha_{\one_{Z(\mu,\nu)}}(\nu x)=\alpha_{\one_{Z(\beta,\gamma)}}(\gamma y)=\beta y$.  Aiming for a contradiction, assume that $|\mu|-|\nu|\neq |\beta|-|\gamma|$.
This, and  the observation above that $\mu\sigma^{|\nu|}(\nu x)=\beta\sigma^{|\gamma|}(\nu x)$ implies that $\nu x$ is eventually periodic, say $\nu x=\kappa\eta\eta\dots$ for some cycle $\eta$ and a $\kappa\in P(E)$ of 
minimal length. Then $\mu x=\lambda\eta\eta\dots$ for some $\lambda\in P(E)$ of minimal length as well.
By Lemma~\ref{lem:getiandj}, there exist $i,j\in\N$ such that $|\lambda\eta^i|-|\kappa\eta^j|=|\mu|-|\nu|$. Similarly, there exist $i',j'\in\N$ such that $|\lambda\eta^{i'}|-|\kappa\eta^{j'}|=|\beta|-|\gamma|$.  
We now have 
\begin{align*}(\mu x,|\mu|-|\nu|,\nu x)&=(\lambda\eta^i\eta\dots, |\lambda\eta^i|-|\kappa\eta^j|, \kappa\eta^j\eta\dots)\text{ and }\\(\beta y,|\beta|-|\gamma|,\gamma y)&=(\lambda\eta^{i'}\eta\dots, |\lambda\eta^{i'}|-|\kappa\eta^{j'}|, \kappa\eta^{j'}\eta\dots).\end{align*}
We may assume that $|\lambda\eta^i|\geq |\mu|$ and that $|\kappa\eta^j|\geq |\nu|$. Now 
\[
\dom(\one_{Z(\mu,\nu)})\cap\dom(\one_{Z(\beta,\gamma)})\supset \dom(\one_{Z(\lambda\eta^i,\kappa\eta^j)})\cap\dom (\one_{Z(\lambda\eta^{i'},\kappa\eta^{j'})})\neq \emptyset.
\] 
By taking an intersection, we may assume that  
\[
\alpha_{\one_{{Z(\lambda\eta^i, \kappa\eta^j)}}}(v)
=\alpha_{\one_{Z(\lambda\eta^{i'},\kappa\eta^{j'})}}(v) \text{\ for all $v\in V$.}
\]
Since $\nu x$ is not isolated, $\eta$ must have an entry, and then there exists a path $\zeta$ with 
$s(\zeta)=s(\eta)$, $|\zeta|\leq|\eta|$ and $\zeta\neq\eta$.  There exists $l\in\N$ such that $l\geq j,j'$ and  $Z(\kappa\eta^l\zeta)\subset Z(\kappa\eta^l)\subset V$. But now for any infinite path $z=\kappa\eta^l\zeta z'\in Z(\kappa\eta^l\zeta)$ we have
\[
\lambda\eta^{l+i-j}\zeta z'=\alpha_{\one_{Z(\lambda\eta^i,\kappa\eta^j)}}(\kappa\eta^l\zeta z')= \alpha_{\one_{Z(\lambda\eta^{i'},\kappa\eta^{j'})}}(\kappa\eta^l\zeta z')=\lambda\eta^{l+i'-j'}\zeta z' .\]   This is a contradiction because $l+i-j\neq l+i'-j'$ and the exit path $\zeta$ does not extend $\mu$.  We have now shown that $\Phi$ is injective.

To see that $\Phi$ is onto, fix $[(n, y)]\in W_E$. Then $y\in\dom (n)$. Write $n=\sum_{(\mu,\nu)\in F}  r_{\mu,\nu}\one_{Z(\mu, \nu)}$ in normal form. Let $x:=\alpha_n(y)$. By Example~\ref{ex alpha n}, there exists $(\beta,\gamma)\in F$ such that $x=\alpha_n(y)=\alpha_{\one_{Z(\beta, \gamma)}}(y)=\beta\sigma^{|\gamma|}(y)$. Now $\sigma^{|\beta|}(x)=\sigma^{|\gamma|}(y)$ implies that $x=\beta x'$ and $y=\gamma x'$. Thus $(x, |\beta|-|\gamma|, y)=(\beta x', |\beta|-|\gamma|, \gamma x')\in G_E$.

First, suppose that $y$ is not isolated.  Let $V=Z(\gamma)$ which is an open neighbourhood of $y$ contained in $\dom(\one_{Z(\beta,\gamma)})\cap \dom(n)$. Let $v\in V$. By Example~\ref{ex alpha n},  $\alpha_n(v)=\beta\sigma^{|\gamma|}(v)=\alpha_{\one_{Z(\beta,\gamma)}}(v)$. Thus 
$\Phi(x, |\beta|-|\gamma|, y)=\Phi(\beta x', |\beta|-|\gamma|, \gamma x')=[(\one_{Z(\beta,\gamma))}, y)]=[(n,y)]$.

Second,  suppose that $y$ is isolated and not periodic.  Then $(\one_{Z(\beta,\gamma)}, y)\sim (n, y)$ because $\alpha_{\one_{Z(\beta,\gamma)}}(y)=x=\alpha_n(y)$ and $p_y n^*\one_{Z(\beta,\gamma)} p_y$ is homogeneous of degree $0$ by Lemma~\ref{lem:isolated}\eqref{it:iso4}. Thus $\Phi(x, |\beta|-|\gamma|, y)=(\beta x', |\beta|-|\gamma|, \gamma y')=[(n,y)]$.

Third, suppose that $y$ is isolated and is periodic. Then $y=\nu\eta\eta\dots$ where $\nu\in P(E)$ and $\eta\in P(E)$ is a cycle without entry.  Then $x=\alpha_n(y)=\mu\eta\eta\dots$ for some $\mu\in P(E)$. By Lemma~\ref{lem:isolated}\eqref{it:iso4}, $p_y(\one_{Z(\mu,\nu)})^*n p_y$ is homogeneous of degree a multiple of $|\eta|$, say $(s-t)|\eta|$ for some $s, t\in\N$. Then 
$((\mu\eta^s)\mu\dots, |\mu|+s|\eta|-(|\nu|+t|\eta|), (\nu\eta^t)\eta\dots)\in G_E$ and 
\[
\Phi(((\mu\eta^s)\mu\dots, |\mu|+s|\eta|-(|\nu|+t|\eta|), (\nu\eta^t)\eta\dots))=[(\mu\eta^s(\nu\eta^t)^*, y)].
\]
We claim that $(\mu\eta^s(\nu\eta^t)^*, y)\sim (n, y)$. We have $\alpha_{\mu\eta^s(\nu\eta^t)^*}(y)=\mu\eta^l\eta\dots=x=\alpha_n(y)$ from above. Let $m=\one_{Z(\nu\eta^t,\mu\eta^s)}\one_{Z(\mu,\nu)}$. Then $\alpha_m(y)=y$, and using Lemma~\ref{lem:alphapx new}\eqref{it:px3} we get
\begin{align*}
\deg(p_y\one_{Z(\nu\eta^t,\mu\eta^s)}n p_y)&=\deg(p_y(p_y\one_{Z(\nu\eta^t,\mu\eta^s)}\one_{Z(\mu,\nu)})
\one_{Z(\nu,\mu)}n p_y)\\
&=\deg(p_y(p_ym)\one_{Z(\nu,\mu)}n p_y)\\
&=\deg(p_ymp_y\one_{Z(\nu,\mu)}n p_y)\\
&=|\nu|+t|\eta|-(|\mu|+s|\eta|)+|\mu|-|\nu|+\deg(p_y\one_{Z(\nu,\mu)}n p_y)\\
&=(t-s)|\eta|+\deg(p_y\one_{Z(\nu,\mu)}n p_y)=0.
\end{align*}
Thus $\Phi(((\mu\eta^s)\mu\dots, |\mu|+s|\eta|-(|\nu|+t|\eta|), (\nu\eta^t)\eta\dots))=[(n,y)]$ when $y$ is isolated and is periodic. We have now shown that $\Phi$ is surjective. Thus $\Phi$ is a groupoid isomorphism.
\end{proof}

\begin{thm}
\label{thm:groupoid iso} Let $E$ be a row-finite directed graph with no sources and let $R$ be an integral domain with identity.  
Then  $\Phi=\Phi_E: G_E\to W_E$ defined by
\[
(\mu x, |\mu|-|\nu|, \nu x)\mapsto [(\one_{Z(\mu,\nu)}, \nu x)]
\]
is a well-defined groupoid isomorphism. Equip $W_E$ with the topology generated by the basis $\{Z(n):n\in N(E)\}$  where $Z(n):=\{[(n,x)]: x\in \dom(n)\}$. Then  $W_E$ is a topological groupoid and is topologically isomorphic to $G_E$. 
\end{thm}

\begin{proof}
In view of Lemmas~ and \ref{lem:themap} and \ref{lem:bij}, it remains  to show  that $\Phi$ is a homeomorphism. 
To see that $\Phi$ is open, we show that the basic open set $Z(\mu, \nu)$ of $G_E$ is mapped to the basic open set $Z(\one_{Z(\mu,\nu)})$ of $W_E$ under $\Phi$:
\begin{align*}
\Phi(Z(\mu,  \nu))&=\{[(\one_{Z(\mu,\nu)},\nu x)]:x\in s(\nu)E^\infty \}\\
&=\{[(\one_{Z(\mu,\nu)},y)]:y\in Z(\nu) \}\\
&=\{[(\one_{Z(\mu,\nu)},y)]:y\in \dom(\one_{Z(\mu,\nu)}) \}\\
&=Z(\one_{Z(\mu,\nu)}).
\end{align*}
Thus $\Phi$ is open.  

To see that $\Phi$ is continuous, suppose that
\[
g_i:=(\mu_i x_i,|\mu_i|-|\nu_i|,\nu_i x_i)\to g:=(\mu  x ,|\mu |-|\nu |,\nu  x )
\]
in $G_E$.  Then  there exist $y_i\in E^\infty$ such that
$\mu_i x_i=\mu y_i$ and $\nu_i x_i=\nu y_i$  eventually, and $\nu y_i\to \nu x$ in $E^\infty$.
Since $|\mu_i|-|\nu_i|=|\mu|-|\nu|$ eventually we must have $g_i=(\mu y_i, |\mu|-|\nu|, \nu y_i)$ eventually. Now let $Z(n)$ be a basic open neighbourhood of $\Phi(g)=[(\one_{Z(\mu,\nu)}, \nu x)]$. It suffices to show that $\Phi(g_i)\in Z(n)$ eventually.  We have $(\one_{Z(\mu,\nu)}, \nu x)\sim (n, \nu x)$.  First suppose that $\nu x$ is isolated.  Then $\nu y_i=\nu x$ eventually, and
\[
\Phi(g_i)=[(\one_{Z(\mu,\nu)}, \nu y_i)]=[(\one_{Z(\mu,\nu)}, \nu x)]\in Z(n).
\]
Second, suppose that $\nu x$ is not isolated.  Then there exists an open neighbourhood $V\subset \dom(n)\cap\dom(\one_{Z(\mu,\nu)})$ of $\nu x$ such that $\alpha_n(w)=\alpha_{\one_{Z(\mu,\nu)}}(w)$ for all $w\in V$. But $\nu y_i\in V$ eventually, and hence
\[
\Phi(g_i)=[(\one_{Z(\mu,\nu)}, \nu y_i)]=[(n, \nu y_i)]\in Z(n)
\]
eventually.  Thus $\Phi$ is continuous. 

We have shown that $\Phi$ is a groupoid isomorphism and a homeomorphism. It follows that composition and  inversion in $W_E$ are  continuous. Hence $W_E$ is a topological groupoid and is isomorphic, as a topological groupoid, to $G_E$.
 \end{proof}


\section{A homeomorphism of infinite-path spaces from Stone duality}\label{section-Stone}
Throughout this section,  $E$ and $F$ are row-finite directed graphs with no sources and  $R$ is an integral domain with identity.
Here we will show that a ring $*$-isomorphism $\HoHo:L_R(E)\to L_R(F)$ 
which  maps the diagonal $D(E)$ onto $D(F)$ induces a homeomorphism $\kappa:E^\infty\to F^\infty$ of infinite path spaces such that
\begin{equation}\label{eq:kappa-again}
\HoHo(\one_L)\circ\kappa=\one_L
\end{equation}
for compact open subsets $L$ of $E^\infty$. If $\pi|_{D(E)}$ is an $R$-algebra homomorphism, then \eqref{eq:kappa-again} holds for all $d\in D(E)$.   We think this result is known but we haven't found a reference for it, and so we give the details here. 

We start by recalling how Stone duality (see, for example, \cite[Theorem~4]{Stone}  and \cite[Chapter~IV]{Burris}) gives a homeomorphism of the spaces of ultrafilters of the  Boolean algebras $B(E)$ and $B(F)$ of idempotents in 
$D(E)$ and $D(F)$.

Let $B=B(\vee, \wedge, 0, \leq)$ be a Boolean algebra; we do not assume $B$ has a unit. A \emph{filter} $U$ of $B$ is a subset of $B$ such that
\begin{enumerate}
\item $x, y\in U$ implies $x\wedge y\in U$;
\item $x\in U$ and $x\leq y$ implies $y\in U$.
\end{enumerate}
An \emph{ultrafilter} is a filter $U$ which is maximal with respect to $0\notin U$. 
Let \[B^*:=\{U\subset B: U \text{\ is an ultrafilter}\}.\]
For $b\in B$, let 
$
N_b:=\{U\in B^*: b\in U\}
$.
Then $\{N_b:b\in B\}$ is a cover of $B^*$, and hence is a subbasis for a topology; we equip $B^*$ with this topology. 

Now let $\pi:B\to C$ be a homomorphism of Boolean algebras.  If $U\in C^*$, then 
$\pi^{-1}(U)$ is in $B^*$, and this gives a function 
\begin{equation}\label{ultrafiltermap}
\tilde\pi:C^*\to B^*\, \quad\tilde\pi(U)=\pi^{-1}(U).
\end{equation}
Since $\pi^{-1}(N_b)=\cup_{c\in\pi^{-1}(\{b\})}N_c$ it follows that $\tilde\pi$ is continuous.  In particular, if $\pi$ is an isomorphism, then  $\tilde\pi$ is a homeomorphism. 

Now let $E$ be a row-finite directed graph with no sources,  let $R$ be an integral domain with identity
and let \[B(E)=\{p\in D(E): p=p^2\}\subset L_R(E).\] Then $B(E)$ is a Boolean algebra with 
$p\vee q=p+q-pq$, $p\wedge q=pq$ and $p\leq q$ if and only if $pq=p$.  Since $R$ is an 
integral domain, every element of $B(E)$ is the characteristic function of a compact open subset $L$ of 
$E^\infty$, and then $\one_L\vee \one_M=\one_{L\cup M}$, $\one_L\wedge \one_M=\one_{L\cap M}$,
and $\one_L\leq \one_M$ if and only if $L\subseteq M$.

\begin{prop}\label{infinitepathspace-uf}
Let $B(E)$ be the Boolean algebra of idempotents of $D(E)$ and $B(E)^*$ the topological space of ultrafilters of $B(E)$. Define $\rho=\rho_E:E^\infty\to B(E)^*$ by
\[
\rho(x)=\{\one_L\in D(E): \text{$L$ is a compact open neighbourhood of $x$}\}.
\]
Then $\rho$ is a homeomorphism with $\rho^{-1}(U)$ the unique element of 
\begin{equation}\label{singleton}
\bigcap_{\one_L\in U}L.
\end{equation}
\end{prop}

\begin{proof} Let $x\in E^\infty$. 
It is straightforward to check that $\rho(x)$ is a filter. To see that it is an ultrafilter, 
suppose $U$ is a filter such that $0 \notin U$  and $\rho(x)\subset U$.  It 
suffices to show $U\subset \rho(x)$. 
Let $\one_L$ be in $U$.  Since $\rho(x)\subset U$, $\one_{Z(x(0,m))}\in U$ for all $m \in \N$. 
Thus $0 \notin U$ implies $L\cap Z(x(0,m))\neq \emptyset$ for all $m \in \N$.  
For each $m \in \N$, choose $x^m\in L\cap Z(x(0,m))$.    Since $Z(x(0,m))$ is a 
neighbourhood base of $x$, $x^m\to x$.  Thus $x \in L$ because $L$ is closed.  
Therefore $1_L\in \rho(x)$ as needed.

To see that $\rho $ is injective, suppose $x,y\in E^\infty$ with $\rho(x)=\rho(y)$. 
Fix $i\in \N$. Then $\one_{Z(x(0,i))}\in\rho(x)=\rho(y)$ implies there exists $j$ 
such that $Z(y(0,j))\subset Z(x(0,i))$. Thus $x(0,i)$ is an initial segment of $y$. Since $i$ was fixed, $x=y$. Thus $\rho$ is injective.

To see that $\rho$ is surjective, let $U\in B(E)^*$. We claim that 
\eqref{singleton}
is a singleton $\{x\}$,  and that $\rho(x)=U$.  Let $\one_M\in U$.
Then 
\[
\bigcap_{\one_L\in U}L=M\cap \bigcap_{\one_L\in U}L=\bigcap_{\one_L\in U}(M\cap L).
\]
Since $U$ is an ultrafilter, $\{M\cap L:\one_L\in U\}$ has the finite intersection property 
in the compact space $M$, and hence the intersection \eqref{singleton} is non-empty.  Let $x\in \bigcap_{\one_L\in U}L$.  
Then $x\in L$ for all $L$ with  $\one_L\in U$. 
By the definition  of $\rho(x)$, $U\subset \rho(x)$, but since $U$ is an ultrafilter, $U=\rho(x)$.  
Notice that injectivity of $\rho$ now shows that the intersection \eqref{singleton} is $\{x\}$.  
We have shown $\rho$ is surjective with inverse $\rho^{-1}(U)$ the unique element of \eqref{singleton}. 

To see that $\rho$ is continuous, it suffices to show that inverse images of the subbasic 
neighbourhoods are open. So fix $\one_L \in B(E)$. Then $N_{\one_L}=\{U\in B(E)^*: \one_L\in U\}$ 
is a typical subbasic neighbourhood. We will show that $\rho^{-1}(N_{\one_L})=L$. First, let $x\in \rho^{-1}(N_{\one_L})$. Then there exists $U\in  N_{\one_L}$ such that $x=\rho^{-1}(U)=\cap_{\one_M\in U} M\subset L$. Thus $x\in L$ and $\rho^{-1}(N_{\one_{L}})\subset L$. Second, let $x\in L$. Then $\one_L\in \rho(x)$, and hence $\rho(x)\in N_{\one_L}$, that is, $x\in \rho^{-1}(N_{\one_L})$. Thus $\rho^{-1}(N_{\one_L})=L$, and since $L$ is open it follows that $\rho$ is continuous.

To see that $\rho$ is open it suffices to show that images of basic neighbourhoods $Z(\mu)$ in $E^\infty$ are open.  We will show that that $\rho(Z(\mu))=N_{1_{Z(\mu)}}$. First let $U\in \rho(Z(\mu))$, that is  $U=\rho(x)$ for some   $x\in Z(\mu)$. Then  $Z(x(0,|\mu|))= Z(\mu)$ gives $\one_{Z(\mu)}\in\ U$. Thus $U\in N_{\one_{Z(\mu)}}$, and hence  $\rho(Z(\mu))\subset N_{1_{Z(\mu)}}$. Second, let $U\in N_{1_{Z(\mu)}}$. Then $1_{Z(\mu)}\in U$. Now $\rho^{-1}(U)=\cap_{\one_L\in U}L\subset Z(\mu)$, that is, $U\in \rho(Z(\mu))$. Thus $\rho(Z(\mu))=N_{1_{Z(\mu)}}$ and is open, and hence $\rho$ is open. 
\end{proof}

\begin{prop}\label{prop-definekappa} Suppose  that $\HoHo: L_R(E)\to L_R(F)$  is a 
ring isomorphism such that $\pi(D(E))= D(F)$.   Then there is a homeomorphism $\kappa: E^\infty\to F^\infty$   such that for every compact open subset $L$ of $E^\infty$ we have 
\begin{equation}\label{kappa-iff}
x\in L\quad\Longleftrightarrow\quad \kappa(x)\in\supp(\pi(\one_L)). 
\end{equation}
Further,   $\supp(d)=\supp(\pi(d)\circ \kappa)$  for all $d\in D(E)$.
\end{prop}

\begin{proof} The ring isomorphism $\pi$ restricts to give an isomorphism of the Boolean algebra $B(E)=\{p\in D(E): p=p^2\}$ onto $B(F)$, and induces a homeomorphism $\tilde\pi: B(F)\to B(E)$ of the corresponding spaces of ultrafilters as given at \eqref{ultrafiltermap}.  Let $\rho_E:E^\infty\to B(E)$ be the homeomorphism of Proposition~\ref{infinitepathspace-uf}.  Then 
\[
\kappa:=\rho_F^{-1}\circ\tilde\pi^{-1}\circ\rho_E
\]
is a homeomorphism such that $\kappa (x)$ is the unique element of \[\bigcap_{\one_M\in\pi(\rho_E(x))}M.\]

To see $\kappa$ satisfies \eqref{kappa-iff}, let $L$ be a compact open subset of $E^\infty$. First, let $x\in L$.  Since $L$ is open, there exists $i\in\N$ such that $Z(x(0,i))\subset L$. Thus $\one_L\in \rho_E(x)$, and $\pi(\one_L)\in \pi(\rho_E(x))$. Now
\[
\{\kappa(x)\}=\bigcap_{\one_M\in\pi(\rho_E(x))}M\subset\supp(\pi(\one_L)). 
\]
Thus $x\in L$ implies $\kappa(x)\in \supp(\pi(\one_L))$.  Second, let $y\in\supp(\pi(\one_L))$. Say $y=\kappa(x)$. For all $M$ such that $\one_M\in \pi(\rho_E(x))$ we have $\kappa(x)\in M$, and hence  $\kappa(x)\in M\cap \supp(\pi(\one_L))$. In other words, for all $\one_N\in\rho_E(x)$ we have $\kappa(x)\in \supp(\pi(\one_N))\cap \supp(\pi(\one_L))$.  Now
\[
0\neq \one_{\supp(\pi(\one_N))\cap \supp(\pi(\one_L))}=\pi(\one_L)*\pi(\one_N)=\pi(\one_L*\one_M)=\pi(\one_{L\cap N}).
\]
Thus $L\cap N\neq\emptyset$ for all such $N$. Because $\rho_E(x)$ is an ultrafilter, the collection $\{L\cap N: \one_N\in\rho_E(x)\}$ has the finite intersection property in the compact space $L$. Thus 
\[
\emptyset\neq \bigcap_{\one_N\in\rho_E(x)}(L\cap N)=L\cap \bigcap_{\one_N\in\rho_E(x)}N=L\cap\{x\}.\]  
Thus $x\in L$ as required. Thus $\kappa$ satisfies \eqref{kappa-iff}. 

For the final assertion, let $d=\sum_{\alpha\in F} r_\alpha\one_{Z(\alpha)}\in D(E)$ in normal form.  Using only that $\pi$ is a ring homomorphism, we have
\begin{align*}
\pi(d) \circ \kappa
&= \sum_{\alpha \in F} \pi(r_\alpha 1_{Z(\alpha)}) \circ \kappa
= \sum_{\alpha \in F} \pi(r_\alpha 1_{Z(\alpha)}^2) \circ \kappa
\\
&= \sum_{\alpha \in F} \big(\pi(r_\alpha 1_{Z(\alpha)}) \pi(1_{Z(\alpha)})\big) \circ \kappa
\\
&= \sum_{\alpha \in F} \big(\pi(r_\alpha 1_{Z(\alpha)}) \circ \kappa\big)  \big(\pi(1_{Z(\alpha)})\circ\kappa\big),
\\
\intertext{which, using  \eqref{kappa-iff}, is just}
&= \sum_{\alpha \in F} \big(\pi(r_\alpha 1_{Z(\alpha)}) \circ \kappa\big)  1_{Z(\alpha)}.
\end{align*}
It follows that  $\supp(\pi(d) \circ \kappa) \subseteq \supp(d)$.  

By symmetry,  $\supp(\pi^{-1}(d') \circ \kappa^{-1}) \subseteq \supp(d')$ for all $d' \in D(F)$. That is,
\[
d'(y) = 0 \implies \pi^{-1}(d')(\kappa^{-1}(y)) = 0\quad\text{ for all $d' \in D(F)$ and $y \in F^\infty$.}
\]
Apply this with  $d' = \pi(d)$ and $y = \kappa(x)$ to get
\[
\pi(d)(\kappa(x)) = 0 \implies \pi^{-1}(\pi(d))(\kappa^{-1}(\kappa(x))) = 0 \implies d(x) = 0.
\]
It follows that $\supp(d) \subseteq \supp(\pi(d)\circ\kappa)$.  Thus $\supp(d) \subseteq \supp(\pi(d)\circ\kappa)$ for all $d\in D(E)$.
\end{proof}

\begin{lemma}\label{lem-nice}
Suppose  that $\HoHo: L_R(E)\to L_R(F)$  is a ring $*$-isomorphism such that $\pi(D(E))=D(F)$. Let $\kappa: E^\infty\to F^\infty$ be the homeomorphism of Proposition~\ref{prop-definekappa}. Then $d=\pi(d)\circ \kappa$ for all $d\in D(E)$ if and only if $\pi|_{D(E)}$ is an $R$-algebra homomorphism.
\end{lemma}

\begin{proof}
First suppose that $d=\pi(d)\circ \kappa$ for all $d\in D(E)$. Let $d,d'\in D(E)$ and $r\in R$.  Then
\[\pi(d+rd')(y)=(d+rd')(\kappa^{-1}(y))=d(\kappa^{-1}(y))+rd'(\kappa^{-1}(y))=(\pi(d)+r\pi(d'))(y)\] and hence $\pi|_{D(E)}$ is an $R$-algebra homomorphism.

Conversely, suppose that $\pi|_{D(E)}$ is an $R$-algebra homomorphism. Let $d\in D$ and write $d=\sum_{\alpha\in F}r_\alpha \one_{Z(\alpha)}$ in normal form. By \eqref{kappa-iff} in Proposition~\ref{prop-definekappa}  we have  $\one_{Z(\alpha)}=\pi(\one_{Z(\alpha)})\circ \kappa$. Thus
\[
\pi(d)\circ \kappa=\big( \sum_{\alpha\in F}r_\alpha \pi(\one_{Z(\alpha)}) \big)\circ\kappa=\sum_{\alpha\in F}r_\alpha \pi(\one_{Z(\alpha)})\circ\kappa= \sum_{\alpha\in F}r_\alpha \one_{Z(\alpha)}=d.\qedhere
\]
\end{proof}

\begin{lemma}
\label{lem:kappa-properties}
Suppose  that $\HoHo: L_R(E)\to L_R(F)$  is a ring $*$-isomorphism such that $\pi(D(E))=D(F)$. 
Let $\kappa: E^\infty\to F^\infty$ be the homeomorphism of Proposition~\ref{prop-definekappa}. 
\begin{enumerate}
\item\label{it:kappa1} If  $x\in E^\infty$ is isolated, then $\HoHo(p_x)=p_{\kappa(x)}$.
\item\label{it:kappa2} Let $n \in N(E)$. Then
\begin{enumerate}
\item \label{subit:kappa2a} $x\in \dom(n)\Longleftrightarrow \kappa(x)\in \dom(\HoHo(n))$; and
\item \label{subit:kappa2b} $\kappa(\alpha_n(x))=\alpha_{\HoHo(n)}(\kappa(x))$ for $x\in \dom(n)$.
\end{enumerate}
\end{enumerate}
\end{lemma}

\begin{proof}  Let $x\in E^\infty$ be an isolated path. Since $\kappa$ is a homeomorphism, $\kappa(x)$ is an isolated path in $F^\infty$.  Recall that $p_x=\one_{\{x\}}$. Applying \eqref{kappa-iff} to the compact open set  $\{x\}$ gives \eqref{it:kappa1}.

For  $n\in N(E)$ we have $n^*n\in D(E)$. Using  Proposition~\ref{prop-definekappa} we have
\begin{align*}x\in\dom(n)=\supp(n^*n)&\Longleftrightarrow 0\neq n^*n(x)\\&\Longleftrightarrow 0\neq \pi(n^*n)(\kappa(x))\\
&\Longleftrightarrow \kappa(x)\in\dom(\pi(n))=\supp(\pi(n)^*\pi(n)),\end{align*}
giving
\eqref{subit:kappa2a}.

For \eqref{subit:kappa2b}, suppose, aiming for a contradiction, that there exists $z\in\dom(n)$ such that 
$\kappa(\alpha_n(z))\neq \alpha_{\HoHo(n)}(\kappa(z))$.  Then there exists a compact open set $M\subset F^\infty$ such that $\alpha_{\HoHo(n)}(\kappa(z)) \in M$ but $\kappa(\alpha_n(z))\notin M$. By \eqref{subit:kappa2a}, $\kappa(z)\in\dom(\pi(n))$. 
 Thus
\begin{align*}
0\neq \one_M(\alpha_{\HoHo(n)}(\kappa(z)))\HoHo(n)^*\HoHo(n)(\kappa(z)) 
&=\big( \one_M\circ\alpha_{\pi(n)}\pi(n)^*\pi(n) \big)(\kappa(z)) \\
&=(\HoHo(n)^*\one_M\HoHo(n))(\kappa(z))\quad\text{(using \eqref{eq:dn2})}\\
&=\HoHo\big( n^*\HoHo\inv (\one_M)n \big)(\kappa(z)).
\end{align*}
By Proposition~\ref{prop-definekappa} we have $0 \neq n^*\pi^{-1}(\one_M)n(z)$, and using \eqref{eq:dn2} again we see
\[0\neq n^*\HoHo\inv (\one_M)n(z)=\HoHo\inv (\one_M)(\alpha_n(z))n^*n(z).\]
Since $R$ is an integral domain we must have $\alpha_n(z)\in \supp(\HoHo\inv (\one_M))$, a contradiction. 
Thus $\kappa(\alpha_n(x))= \alpha_{\HoHo(n)}(\kappa(x))$ for all $x\in\dom(n)$, giving (\ref{subit:kappa2b}).
\end{proof}

\section{The main theorem}\label{sec-main} Throughout this section,  $E$ and $F$ are row-finite directed graphs with no sources and  $R$ is an integral domain with identity. 
In this section we show that a diagonal-preserving ring $*$-isomorphism of $L_R(E)$ onto $L_R(F)$  
induces an isomorphism of the groupoid $G_E$ onto $G_F$. It then follows, for example, that the $C^*$-algebras $C^*(E)$ and $C^*(F)$ are isomorphic. 

\begin{prop}
\label{prop:diagiso}
Suppose  that $\HoHo: L_R(E)\to L_R(F)$  is a ring $*$-isomorphism  such that $\pi(D(E))=D(F)$. 
Let $\kappa: E^\infty\to F^\infty$ be the homeomorphism of Proposition~\ref{prop-definekappa}.
Then $\Psi=\Psi_{\HoHo}:W_E\to W_F$ defined by 
$
\Psi([(n,x)])=[(\HoHo(n),\kappa(x))]
$
is an isomorphism of the Weyl groupoids.
\end{prop}

\begin{proof}
Since $\HoHo$ is $*$- and diagonal-preserving by assumption,  $\HoHo$ maps the normalizers of $D(E)$ to the normalizers of $D(F)$. If $x\in\dom(n)$ then $\kappa(x)\in\dom(\HoHo(n))$ by Lemma~\ref{lem:kappa-properties}\eqref{subit:kappa2a}, and the formula for $\Psi$ makes sense.  

To see $\Psi$ is well-defined, suppose that $(n_1,x)\sim (n_2,x)$. We need to show that $(\HoHo(n_1),\kappa(x))\sim (\HoHo(n_2),\kappa(x))$. Since  $\kappa$ is a homeomorphism, $x$ is isolated in $E^\infty$  if and only if $\kappa(x)$ is isolated in $F^\infty$.

First, suppose that $x$ is not isolated. Since $(n_1,x)\sim (n_2,x)$, there exists a neighbourhood $V\subset \dom(n_1)\cap \dom(n_2)$ of $x$  such that $\alpha_{n_1}(z)=\alpha_{n_2}(z)$ for all $z\in V$.  Then $\kappa(V)$ is an open neighbourhood of $\kappa(x)$, and is contained in $\dom(\HoHo(n_1))\cap \dom(\HoHo(n_2))$ by Lemma~\ref{lem:kappa-properties}\eqref{subit:kappa2a}. Let $y\in\kappa(V)$. Then $y=\kappa(z)$ for some $z\in V$ and using Lemma~\ref{lem:kappa-properties}\eqref{subit:kappa2b} we have
\[
\alpha_{\HoHo(n_1)}(y)=\alpha_{\HoHo(n_1)}(\kappa (z))=\kappa(\alpha_{n_1}(z))=\kappa (\alpha_{n_2} (z))=\alpha_{\HoHo(n_2)}(\kappa(z))=\alpha_{\HoHo(n_2)}(y).
\]
Thus $(\HoHo(n_1),\kappa(x))\sim (\HoHo(n_2),\kappa(x))$ if $x$ is not isolated.

Second, suppose that $x$ is  isolated. Then $\alpha_{n_1}(x)=\alpha_{n_2}(x)$, and it follows again from Lemma~\ref{lem:kappa-properties}\eqref{subit:kappa2b}  that $\alpha_{\HoHo(n_1)}(\kappa(x))=\alpha_{\HoHo(n_2)}(\kappa(x))$. Also, $p_x n_1^* n_2p_x$ is $0$-graded.
Apply Lemma~\ref{jackpot} with $a=p_xn_1^*n_2p_x$ to get
\[p_xn_1^*n_2p_x = p_xap_x = rp_x \in D(E).\]
Now
\[\HoHo(p_xn_1^*n_2p_x) = \HoHo(rp_x) \in D(F)\]
because $\HoHo$ is diagonal preserving.  Thus
$\HoHo(p_xn_1^*n_2p_x) = p_{\kappa(x)}\HoHo(n_1)^*\HoHo(n_2)p_{\kappa(x)}$ is $0$-graded.
Therefore $(\HoHo(n_1),\kappa(x))\sim (\HoHo(n_2),\kappa(x))$ when $x$ is  isolated.  
This proves that $\Psi$ is well-defined. 

The map $[(m,z)]\mapsto [(\HoHo\inv(m),\kappa\inv(z))]$ is clearly an inverse for $\Psi$, so $\Psi$ is a bijection. (That $\Psi^{-1}$ is well-defined follows from the same calculation used to show that $\Psi$ is well-defined.)
The topology on $W_E$ has basic open sets $Z(n):=\{[(n,x)]: x\in \dom(n)\}$.  Here $\Psi(Z(n))=Z(\HoHo(n))$ and $\Psi\inv(Z(m))=Z(\HoHo\inv(m))$. Thus $\Psi$ is both continuous and open, and hence is a homeomorphism.
Finally, we check  that $\Psi$ is a homomorphism:
\begin{align*}
\Psi([(n_1,\alpha_{n_2}(x))][(n_2,x)]) &=\Psi([(n_1n_2, x)])\\
&=[(\HoHo(n_1n_2), \kappa(x))]\\
&=[(\HoHo(n_1), \alpha_{\HoHo(n_2)}(\kappa(x)))][(\HoHo(n_2), \kappa(x))]\\
&=[(\HoHo(n_1), \kappa(\alpha_{n_2}(x)))][(n_2, \kappa(x))]\\
&=\Psi([(n_1,\alpha_{n_2}(x))])\Psi([(n_2,x)]).\qedhere
\end{align*}
\end{proof}

\begin{thm}
\label{thm:main2}
Let $E$ and $F$ be row-finite directed graphs with no sources, and let $R$ be an integral domain with identity. The following are equivalent.
\begin{enumerate}
\item\label{main2-item1} There is a ring $*$-isomorphism $\HoHo: L_R(E)\to L_R(F)$ such that $\pi(D(E))=D(F)$.
\item\label{main2-item2}  The topological groupoids $G_E$ and $G_F$ are isomorphic. 
\end{enumerate}
\end{thm}

\begin{proof} Assume \eqref{main2-item1}. Let $\Phi_E: G_E\to W_E$ be the isomorphism of topological groupoids from Theorem~\ref{thm:groupoid iso} and let $\Psi_{\HoHo}: W_E\to W_F$ be the isomorphism of topological groupoids from Proposition~\ref{prop:diagiso}. Then $\Phi_F\inv\circ\Psi_{\HoHo}\circ\Phi_E: G_E\to G_F$ is an isomorphism of topological groupoids. 

Assume \eqref{main2-item2}.  Say $\Omega:G_E\to G_F$ is a topological isomorphism.  Let $\HoHo: L_R(E)\to L_R(F)$ be $\HoHo(f)=f\circ\Omega\inv$; then  $\HoHo$ is an algebra $*$-isomorphism, and hence is in particular  a ring $*$-isomorphism. 
Since $\Omega$ maps the unit space $E^\infty$ of $G_E$ to $F^\infty$, it follows that $\pi$ preserves the diagonal. 
\end{proof}

\begin{cor}\label{cor1} Let $E$ and $F$ be row-finite directed graphs with no sources, and let $R$ be an integral domain and let $S$ be a commutative ring, 
both  with identities. Suppose that there is a ring $*$-isomorphism $\HoHo: L_R(E)\to L_R(F)$ such that $\pi(D(E))=D(F)$.   Then there exists an algebra $*$-isomorphism of $L_S(E)$ onto $L_S(F)$.  
Further,  there is a $C^*$-algebra isomorphism of $C^*(E)$ onto $C^*(F)$. 
\end{cor}

\begin{proof}  By Theorem~\ref{thm:main2}, the groupoids $G_E$ and $G_F$ are topologically isomorphic. Hence their Leavitt path algebras over any commutative ring $S$ with identity are isomorphic as $*$-algebras, and similarly the $C^*$-algebras. 
\end{proof}

\begin{cor} Let $E$ and $F$ be row-finite directed graphs with no sources, and let $R$ be an integral domain with identity. Suppose  that $\HoHo: L_R(E)\to L_R(F)$  is a ring $*$-isomorphism  such that $\pi(D(E))=D(F)$. Then $d=\pi(d)\circ \kappa$ for all $d\in D(E)$. 
\end{cor}

\begin{proof} By Corollary~\ref{cor1} there exists an $R$-algebra $*$-isomorphism of $L_R(E)$ onto $L_R(F)$. Its restriction to the diagonal is an $R$-algebra homomorphism. Thus  $d=\pi(d)\circ \kappa$ for all $d\in D(E)$ by Lemma~\ref{lem-nice}.
\end{proof}

As mentioned in the introduction, taking $R=\C$ in Corollary~\ref{cor1} gives some evidence towards the ``isomorphism conjecture for graph algebras'' from \cite[page~3752]{AT}.

\begin{cor}
\label{cor2}
Let $E$ and $F$ be row-finite directed graphs with no sources and let $R$ be an integral domain. The following are equivalent.
\begin{enumerate}
\item\label{cor2-item1} There is a ring $*$-isomorphism $\pi: L_R(E)\to L_R(F)$ such that $\pi(D(E))=D(F)$.
\item\label{cor2-item2} The topological groupoids $G_E$ and $G_F$ are isomorphic. 
\item\label{cor2-item3}   There is a $*$-algebra isomorphism of $L_\Z(E)$ onto $L_\Z(F)$. 
\item\label{cor2-item4}  There is a $C^*$-algebra isomorphism of $C^*(E)$ onto $C^*(F)$ which preserves the $C^*$-diagonals.
\end{enumerate}
\end{cor}
\begin{proof} 
By \cite[Theorem~1]{Toke}, \eqref{cor2-item2}--\eqref{cor2-item4} are equivalent, and  \eqref{main2-item1} and \eqref{cor2-item2} are equivalent by  Theorem~\ref{thm:main2}.
\end{proof}

\end{document}